\documentclass[12pt]{amsart}

\font\bbbld=msbm10 scaled\magstephalf

\newcommand{\bM}{\bar{M}}

\newcommand{\bfR}{\hbox{\bbbld R}}
\newcommand{\bfS}{\hbox{\bbbld S}}

\newcommand{\tF}{\tilde{F}}

\newcommand{\ol}{\overline}
\newcommand{\ul}{\underline}

\newtheorem{theorem}{Theorem}[section]
\newtheorem{lemma}[theorem]{Lemma}
\newtheorem{proposition}[theorem]{Proposition}

 \theoremstyle{definition}

\theoremstyle{remark}
\newtheorem{remark}[theorem]{Remark}

\numberwithin{equation}{section}

\begin{document}

\title[fully nonlinear elliptic equations]
{
Second order estimates for Hessian type \\ fully nonlinear elliptic
equations\\ on Riemannian manifolds}
\author{Bo Guan}
\address{Department of Mathematics, Ohio State University,
         Columbus, OH 43210, USA}
\email{guan@math.osu.edu}
\author{Heming Jiao}
\address{Department of Mathematics, Harbin Institute of Technology,
         Harbin, 150001, China}
\email{jiaoheming@163.com}
\thanks{Research of the first author was supported in part by NSF grants
and MRI at OSU.
Research of the second author was supported in part
by a CRC graduate fellowship.}

\begin{abstract}

We derive  {\em a priori} estimates for second order derivatives of
solutions to a wide calss of fully nonlinear elliptic equations on 
Riemannian manifolds. The equations we consider naturally 
appear in geometric problems and other applications such as optimal transportation.
There are some fundamental assumptions in the literature to ensure the equations to be elliptic and that one can apply Evans-Krylov theorem 
once the $C^2$ estimates are derived. 
However, in previous work one needed extra assumptions which are
more technical in nature to overcome various difficulties. 
In this paper we are able to remove most of these technical assumptions. 
Indeed, we derive the estimates under conditions which are almost optimal, and prove existence results for the Dirichlet problem which 
are new even for bounded domains in Euclidean space. 
Moreover, our methods can be applied to other types 
of nonlinear elliptic and parabolic equations, including those on complex manifolds.

\end{abstract}

\maketitle

\section{Introduction}

\medskip

In the study of fully nonlinear elliptic or parabolic
equations, {\em a priori} $C^2$ estimates are crucial to the question of
existence and regularity of solutions. Such estimates are also important
in applications.
In this paper we are concerned with second derivative estimates
for solutions of the Dirichlet problem for
equations of the form
\begin{equation}
\label{3I-10}
f (\lambda (\nabla^2 u + A [u]))
  = \psi (x, u, \nabla u)
\end{equation}
on a Riemannian manifold $(M^n, g)$ of dimension $n \geq 2$ with smooth boundary $\partial M$,
with boundary condition
\begin{equation}
\label{3I-10d}
u = \varphi \;\; \mbox{on $\partial M$},
\end{equation}
where $f$ is a symmetric smooth function of $n$ variables,
$\nabla^2 u$ is the Hessian of $u$,
$A [u] = A (x, u, \nabla u)$ a $(0,2)$ tensor which may depend on $u$
and $\nabla u$,
and
\[ \lambda (\nabla^2 u + A [u]) = (\lambda_1 ,\ldots,\lambda_n) \]
denotes the eigenvalues of $\nabla^2 u + A [u]$ with respect to the
metric $g$.

Following the pioneer work of Caffarelli, Nirenberg and
Spruck~\cite{CNS3}
we assume $f$ to be defined in an open, convex, symmetric
cone $\Gamma \subset \mathbb{R}^{n}$ with vertex at the origin,
\[ \Gamma_n \equiv\{\lambda\in\mathbb{R}^{n}:
\mbox{each component } \lambda_{i}>0\} \subseteq \Gamma \neq \mathbb{R}^{n} \]
and to satisfy the standard structure conditions in the literature:
\begin{equation}
\label{3I-20}
f_{i} \equiv \frac{\partial f}{\partial \lambda_{i}} > 0
\mbox{ in }\Gamma,\  1 \leq i \leq n,
\end{equation}
\begin{equation}
\label{3I-30}
\mbox{$f$ is a concave function in $\Gamma$},
\end{equation}
and 
\begin{equation}
\label{3I-40}
\sup_{\partial \Gamma} f  \equiv \sup_{\lambda_0 \in \partial \Gamma}
                   \limsup_{\lambda \rightarrow \lambda_0} f (\lambda) \leq 0.
\end{equation}

We shall call a function $u \in C^{2}(\bM)$ {\em admissible} if
$\lambda (\nabla^{2} u + A [u]) \in \Gamma$.
By \eqref{3I-20}, equation \eqref{3I-10} is elliptic for admissible
solutions.

When $A =0$ or $A (x)$, the Dirichlet problem~\eqref{3I-10}-\eqref{3I-10d}
in $\bfR^n$ for was first studied by Ivochkina~\cite{Ivochkina80}
and Caffarelli, Nirenberg and Spruck~\cite{CNS3}, followed by
work in \cite{LiYY90}, \cite{Wang94}, \cite{Guan94}, \cite{Trudinger95},
\cite{TW99}, and \cite{CW01}, etc. Li~\cite{LiYY90} and Urbas~\cite{Urbas02}
studied equation~\eqref{3I-10} with $A = g$ on closed Riemannian manifolds;
see also \cite{Guan99a} where the Dirichlet
problem was treated for $A = \kappa u g$ ($\kappa$ is constant).

A critical issue to solve the Dirichlet problem for equation~\eqref{3I-10},
is to derive {\em a priori} $C^2$ estimates for admissible solutions.
By conditions \eqref{3I-20} and \eqref{3I-40},  equation~\eqref{3I-10} becomes
uniformly elliptic once $C^2$ estimates are established,
and one therefore
 obtains global $C^{2, \alpha}$ estimates using Evans-Krylov theorem
which crucially relies on the concavity condition \eqref{3I-30}.
From this point of view, conditions~\eqref{3I-20}-\eqref{3I-40} are
fundamental to the classical solvability of equation~\eqref{3I-10}.

In this paper we shall primarily focus on deriving {\em a priori} estimates
for second order derivatives.
In order to state our main results let us introduce some notations;
see also \cite{Guan12a}.

For $\sigma > 0$ let $\Gamma^{\sigma}
    = \{\lambda \in \Gamma: f (\lambda) > \sigma\}$
which we assume to be nonempty.
By assumptions~(\ref{3I-20}) and (\ref{3I-30}) the boundary of
$\Gamma^{\sigma}$,
$\partial \Gamma^{\sigma} = \{\lambda \in \Gamma: f (\lambda) = \sigma\}$ is
a smooth, convex and noncompact hypersurface in $\bfR^n$.
For $\lambda \in \Gamma$ we use
$T_{\lambda} = T_{\lambda} \partial \Gamma^{f (\lambda)}$ to denote
the tangent plane at $\lambda$ to the level surface
$\partial \Gamma^{f (\lambda)}$.

The following new condition is essential to our work in this paper:
\begin{equation}
\label{gj-I100}
\mbox{$\partial \Gamma^{\sigma} \cap
    T_{\lambda} \partial \Gamma^{f (\lambda)} $ is nonempty and compact},
\; \forall \, \sigma > 0, \; \lambda \in \Gamma^{\sigma}.
\end{equation}

Throughout the paper we assume $\bM := M \cup \partial M$ is
compact and $A [u]$ is smooth on $\bM$ for $u \in C^{\infty} (\bM)$,
$\psi \in C^{\infty} (T^*\bM \times \bfR)$ (for convenience we shall write
$\psi = \psi (x, z, p)$ for $(x, p) \in T^*\bM$ and $z \in \bfR$ though),
$\psi > 0$,
$\varphi \in C^{\infty} (\partial M)$.
Note that for fixed $x \in \bM$, $z \in \bfR$ and $p \in T^*_x M$,
\[ A (x, z, p): T^*_x M \times T^*_x M \rightarrow \bfR \]
is a symmetric bilinear map. We shall use the notation
\[ A^{\xi \eta} (x, \cdot, \cdot) := A (x, \cdot, \cdot) (\xi, \eta), \;\;
\xi,  \eta \in T^*_x M \]
and, for a function $v \in C^2 (M)$,  $A [v] := A (x, v, \nabla v)$,
$A^{\xi \eta} [v] := A^{\xi \eta} (x, v, \nabla v)$.

\begin{theorem}
\label{gj-th2}
Assume, in addition to \eqref{3I-20}-
\eqref{gj-I100}, that
\begin{equation}
\label{A2}
\mbox{$-\psi (x,z,p)$ and $A^{\xi \xi} (x,z,p)$ are concave
   in $p$},
\end{equation}
\begin{equation}
\label{A4}
- \psi_z, \; A^{\xi \xi}_{z}  \geq 0,
\;\; \forall \, \xi \in T_x M.
\end{equation}
and that there exists
an admissible subsolution $\ul{u} \in C^2 (\bM)$ satisfying
\begin{equation}
\label{3I-11s}
\left\{
\begin{aligned}
f (\lambda (\nabla^2 \ul{u} + A [\ul u]))
    \,& \geq \psi (x, \ul u, \nabla \ul u) \;\; \mbox{in $\bM$}, \\
    \ul u \,& = \varphi \;\; \mbox{on $\partial M$}.
    \end{aligned} \right.
\end{equation}
Let $u \in C^4 (\bM)$ be an admissible solution of
equation~\eqref{3I-10} with $u \geq \ul u$ on $\bM$.
Then
\begin{equation}
\label{hess-a10}
\max_{\bM} |\nabla^2 u| \leq
 C_2 \big(1 + \max_{\partial M}|\nabla^2 u|\big)
\end{equation}
where $C_2 > 0$ depends on $|u|_{C^1 (\bM)}$ and $|\ul u|_{C^2 (\bM)}$.
In particular, if $M$ is closed, i.e. $\partial M = \emptyset$, then
\begin{equation}
\label{hess-a10c}
|\nabla^2 u| \leq C_2 \;\; \mbox{on $M$}.
\end{equation}

Suppose that $u$ also satisfies the boundary condition
\eqref{3I-10d} and that
\begin{equation}
\label{3I-45}
\sum f_{i} (\lambda) \lambda_{i} \geq 0,
     \; \forall \, \lambda \in \Gamma.
\end{equation}
Then there exists $C_3 > 0$ depending on
$|u|_{C^1 (\bM)}$, $|\ul u|_{C^2 (\bM)}$ and
$|\varphi|_{C^4 (\partial M)}$ such that
\begin{equation}
\label{hess-a10b}
\max_{\partial M}|\nabla^2 u| \leq C_3.
\end{equation}
\end{theorem}

The assumption~\eqref{gj-I100} excludes linear elliptic equations but
is satisfied by a very general class of functions $f$.
In particular, Theorem~\ref{gj-th2}
applies to $f = \sigma_k^{\frac{1}{k}}$, $k \geq 2$ and
$f = (\sigma_k/\sigma_l)^{\frac{1}{k-l}}$, $1 \leq l <  k \leq n$ in
the Garding cone
\[  \Gamma_k = \{\lambda \in \bfR^n: \sigma_j (\lambda) > 0 ,\;
\mbox{$\forall$ $1 \leq j \leq k$}\}, \]
where $\sigma_k$ is the $k$-th elementary symmetric
function; see \cite{Guan12a}.
It is also straightforward to verify that the function $f = \log P_k$
satisfies assumptions~\eqref{3I-20}-
\eqref{gj-I100} where
\[ P_k (\lambda) := \prod_{i_1 < \cdots < i_k}
(\lambda_{i_1} + \cdots + \lambda_{i_k}), \;\; 1 \leq k \leq n\]
defined in the cone
\[ \mathcal{P}_k : = \{\lambda \in \bfR^n:
      \lambda_{i_1} + \cdots + \lambda_{i_k} > 0\}. \]

When both $A$ and $\psi$ are independent of $u$ and $\nabla u$,
Theorem~\ref{gj-th2} was proved by the first
author~\cite{Guan12a} under the weaker assumption
that \eqref{gj-I100} holds for all $\lambda \in \partial \Gamma^{\sigma}$,
which improves previous results
due to Caffarelli, Nirenberg and Spruck~\cite{CNS3}, Li~\cite{LiYY90},
Trudinger~\cite{Trudinger95}, Urbas~\cite{Urbas02} and the first
author~\cite{Guan94}, etc. Clearly, the two conditions are equivalent
if $f$ is homogeneous or more generally
$f (t \lambda ) = h (t) f (\lambda)$, $\forall \, t > 0$  in $\Gamma$
for some positive function $h$.

Some of the major difficulties in deriving the estimates
\eqref{hess-a10} and \eqref{hess-a10b} are caused
by the presence of curvature and lack of good globally defined functions
on general Riemannian manifolds, and by the arbitrary geometry of boundary.
As in \cite{Guan12a}, we make crucial use of the subsolution
in both estimates to overcome the difficulties. (We shall also use
key ideas from \cite{CNS3}, \cite{LiYY90},
\cite{Trudinger95}, \cite{Urbas02}, etc. all of which contain significant
contributions to the subject.)
However, the proofs in the
current paper are far more delicate than those in \cite{Guan12a},
especially for the boundary estimates.
The core of our approach is the following inequality

\begin{theorem}
\label{3I-th3}
Assume that \eqref{3I-20}, \eqref{3I-30} and
\eqref{gj-I100} hold.
Let $K$ be a compact subset of $\Gamma$ and 
$\sup_{\partial \Gamma}  f < a \leq b < \sup_{\Gamma}  f$.
There exist positive constants
$\theta = \theta (K, [a,b])$ and $R = R (K, [a,b])$ such that for any 
$\lambda \in \Gamma^{[a,b]} = \ol{\Gamma^a} \setminus \Gamma^b$, 
when $|\lambda| \geq R$,
\begin{equation}
\label{3I-100a'}
\sum f_i (\lambda) (\mu_{i} - \lambda_i) \geq
     \theta + \theta \sum f_i (\lambda) + f (\mu) - f (\lambda),
\;\; \forall \, \mu \in K.
\end{equation}
\end{theorem}

Perhaps the most important contribution of this paper is the new idea 
introduced in the proof of Theorem~\ref{3I-th3}. 
It will be further developed in our forthcoming work (e.g. \cite{Guan}). 
Note that by the concavity of $f$ we always have
\[ \sum f_i (\lambda) (\mu_{i} - \lambda_i) \geq f (\mu) - f (\lambda),
 \;\; \forall \, \mu, \lambda \in \Gamma. \]
It may also be worthwhile to point out that in Theorem~\ref{3I-th3} the
function $f$ is {\em not} assumed to be {\em strictly} concave.

In general, without assumption~\eqref{3I-11s} the Dirichlet problem
for equation~\eqref{3I-10} is not always solvable either if
$A$ or $\psi$ is dependent on $u$ and $\nabla u$,
or if there is no geometric restrictions to $\partial M$ being imposed.

The recent work of Guan, Ren and Wang~\cite{GRW13} shows that the convexity
assumption on $\psi$ in $\nabla u$ can not be dropped in general from
Theorem~\ref{gj-th2}. On the other hand, they
derived the second order estimates for $f = \sqrt{\sigma_2}$
without the assumption;
such estimates are also known to hold for the Monge-Amp\`ere equation
($f = \sigma_n^{1/n}$).
It seems an interesting open question whether it is still true for
$f = \sigma_k^{1/k}$, $3 \leq k < n$; see \cite{GRW13}.

Our motivation to study equation~\eqref{3I-10} comes in part
from its natural connection to geometric problems, and
the problem of optimal transportation which turns out to
be very closely related and to have interesting applications to
differential geometry.
The potential function of an optimal mass transport satisfies
a Monge-Amp\`ere type equation of form \eqref{3I-10}
where $f = \sigma_n^{1/n}$ and $A$ is determined by the cost function.
In \cite{MTW05} Ma, Trudinger and Wang introduced the following condition
to establish interior regularity for optimal transports:
there exists $c_0 > 0$ such that
\begin{equation}
\label{A3}
A^{\xi \xi}_{p_k p_l} (x,z,p) \eta_k \eta_l \leq - c_0 |\xi|^{2}|\eta|^{2}, \;
\forall \, \xi, \eta \in T_x M, \; \xi \perp \eta,
\end{equation}
now often referred as the MTW condition.  In this paper we derive the
following interior estimate which also extends Theorem 2.1 in
\cite{Trudinger06}.

\begin{theorem}
\label{gj-th4}
In addition to \eqref{3I-20}-\eqref{3I-40}, \eqref{A2} and \eqref{A3},
assume that
\begin{equation}
\label{gj-I105}
 \lim_{t \rightarrow \infty} f (t {\bf 1}) = + \infty
\end{equation}
where ${\bf 1}= (1, \ldots, 1) \in \bfR^n$.
Let $\ul{u} \in C^2 (\bar{B_r})$ be an admissible function
and $u \in C^4 (B_r)$ an admissible solution of \eqref{3I-10}
 in a geodesic ball $B_r \subset M$ of radius $r > 0$. Then
\begin{equation}
\label{gj-G185}
\sup_{B_{\frac{r}{2}}} |\nabla^2 u| \leq C_4
\end{equation}
where $C_4$ depends on $r^{-1}$, $|u|_{C^1 (B_r)}$,
$|\ul u|_{C^2 (B_r)}$, and other known data.
\end{theorem}

\begin{remark}
The function $\ul{u} \in C^2 (\bar{B_r})$ does not have to be a subsolution.
\end{remark}

Equations of form \eqref{3I-10} appear in many
interesting geometric problems. These include
the Minkowski problem (\cite{Nirenberg53},  \cite{Pogorelov52},
\cite{Pogorelov78},
\cite{CY76b}) and its generalizations proposed by
Alexandrov~\cite{Alexandrov56}
and Chern~\cite{Chern59} (see also \cite{GG02}),
the Christoffel-Minkowski problem (cf. eg. \cite{GM03}),
and the Alexandrov problem of prescribed curvature measure
(cf, e.g. \cite{GL97}, \cite{GLL12}),
which are associated with equation~\eqref{3I-10}
 on $\bfS^n$
for $f = \sigma_n^{\frac{1}{n}}, \, (\sigma_n/\sigma_l)^{\frac{1}{n-l}}$
or $\sigma_k^{\frac{1}{k}}$ and $A = u g$.
Another classical example is the Darboux equation
\begin{equation}
\label{gj-I45}
\det (\nabla^2 u + g) = K (x) (-2 u - |\nabla u|^2) \det g
\end{equation}
on a positively curved surface $(M^2, g)$, which appears
in isometric embedding,
e.g. the Weyl problem (\cite{Nirenberg53}, \cite{GL94}, \cite{HZ95}).
In \cite{GW98} Guan and Wang studied the Monge-Amp\`ere type equation on $\bfS^n$
\begin{equation}
\label{gj-I50}
\det \Big(\nabla^2 u - \frac{u^2 + |\nabla u|^2}{2 u} g\Big)
= K (x) \Big(\frac{u^2 + |\nabla u|^2}{2 u}\Big)^n \det g
\end{equation}
which arises from reflector antenna designs in engineering,
while the Schouten tensor equation
\begin{equation}
\label{gj-I60}
\sigma_k \Big(\nabla^2 u + d u \otimes du
    - \frac{1}{2} |\nabla u|^2 g + S_g\Big) = \psi (x) e^{-2ku}
\end{equation}
(where $S_g$ is the Schouten tensor of $(M^n, g)$)
introduced by Viaclovsky~\cite{Viaclovsky00} is connected with a natural
fully nonlinear version of the Yamabe problem and has very interesting
applications in conformal geometry; see for instance \cite{CGY02a},
\cite{GeW06}, 
\cite{GW03b}, 
\cite{GV07}, 
\cite{LL03}, 
\cite{STW07}
and references therein.

Interior second order estimates were obtained in \cite{GW98} for
equation~\eqref{gj-I50} and in  \cite{GW03a} for equation~\eqref{gj-I60};
see also \cite{ChenS05} for a simplified proof.
The reader is also referred to \cite{Pogorelov78} for the classical
Pogorelov estimate for the Monge-Amp\`ere equation,
and to \cite{CW01}, \cite{SUW04} for its generalizations to the Hessian
equation ($f = \sigma_k^{1/k}$ in \eqref{3I-10})
and equations of prescribed curvature, respectively.

The rest of this paper is organized as follows.
In Section~\ref{gj-P} we present a brief review of some elementary
formulas and a consequence of Theorem~\ref{3I-th3};
see Proposition~\ref{gj-lem3}.
In Section~\ref{gj-C} we give a proof of Theorem~\ref{3I-th3}.
In Section~\ref{gj-S} we prove the maximum principle \eqref{hess-a10}
in Theorem~\ref{gj-th2} and derive the interior estimate \eqref{gj-G185},
while the boundary estimate \eqref{hess-a10b} is established in
Section~\ref{gj-B}.
Section~\ref{gj-G} is devoted to the gradient estimates.
Finally in Section~\ref{gj-E} we state some existence results for
the Dirichlet problem~\eqref{3I-10}-\eqref{3I-10d} based on the gradient
and second order estimates in Theorem~\ref{gj-th2} and Section~\ref{gj-G},
which can be proved using the standard method of continuity and degree theory.

Part of this work was done while the second author was visiting
Department of Mathematics at Ohio State University. He wishes to
thank the Department and the University for their hospitality.
The first author wishes to thank Xinan Ma for bringing the functions $P_k$ to
his attention in a conversation several years ago.

\bigskip

\section{Preliminaries}
\label{gj-P}
\setcounter{equation}{0}

\medskip

Throughout the paper let $\nabla$ denote the Levi-Civita connection
of $(M^n, g)$. The curvature tensor is defined by
\[ R (X, Y) Z = - \nabla_X \nabla_Y Z + \nabla_Y \nabla_X Z
                 + \nabla_{[X, Y]} Z. \]
Under a local frame $e_1, \ldots, e_n$ on $M^n$ we denote
$g_{ij} = g (e_i, e_j)$, $\{g^{ij}\} = \{g_{ij}\}^{-1}$, while
the Christoffel symbols $\Gamma_{ij}^k$ and curvature
coefficients are given respectively by
$\nabla_{e_i} e_j = \Gamma_{ij}^k e_k$ and
\[  R_{ijkl} = g( R (e_k, e_l) e_j, e_i), \;\; R^i_{jkl} = g^{im} R_{mjkl}.  \]
We shall use the notation $\nabla_i = \nabla_{e_i}$,
$\nabla_{ij} = \nabla_i \nabla_j - \Gamma_{ij}^k \nabla_k $, etc.

For a differentiable function $v$ defined on $M^n$, we usually
identify $\nabla v$ with its gradient, and use
$\nabla^2 v$ to denote its Hessian which is locally given by
$\nabla_{ij} v = \nabla_i (\nabla_j v) - \Gamma_{ij}^k \nabla_k v$.
We recall that $\nabla_{ij} v =\nabla_{ji} v$ and
\begin{equation}
\label{hess-A70}
 \nabla_{ijk} v - \nabla_{jik} v = R^l_{kij} \nabla_l v,
\end{equation}
\begin{equation}
\label{hess-A80}
\begin{aligned}
\nabla_{ijkl} v - \nabla_{klij} v
= R^m_{ljk} \nabla_{im} v & + \nabla_i R^m_{ljk} \nabla_m v
      + R^m_{lik} \nabla_{jm} v \\
  & + R^m_{jik} \nabla_{lm} v
      + R^m_{jil} \nabla_{km} v + \nabla_k R^m_{jil} \nabla_m v.
\end{aligned}
\end{equation}

Let $u \in C^4 (\bM)$ be an admissible solution of
equation~\eqref{3I-10}.
For simplicity we shall denote $U := \nabla^2 u + A (x, u, \nabla u)$
and, under a local frame $e_1, \ldots, e_n$,
\[ U_{ij} \equiv U (e_i, e_j) = \nabla_{ij} u + A^{ij} (x, u, \nabla u), \]
\begin{equation}
\label{gj-P10}
\begin{aligned}
\nabla_k U_{ij}
  \equiv \,& \nabla U (e_i, e_j, e_k)
             = \nabla_{kij} u + \nabla_k A^{ij} (x, u, \nabla u)  \\
  \equiv \,& \nabla_{kij} u  + \nabla'_k A^{ij} (x, u, \nabla u)
              + A^{ij}_z (x, u, \nabla u) \nabla_k u \\
           & + A^{ij}_{p_l} (x, u, \nabla u) \nabla_{kl} u
 \end{aligned}
\end{equation}
 where $A^{ij} = A^{e_{i} e_{j}}$ and
$\nabla'_k A^{ij}$ denotes the {\em partial} covariant derivative of $A$
when viewed as depending on $x \in M$ only, while the meanings of
$A^{ij}_z$ and $A^{ij}_{p_l}$, etc are obvious. Similarly we can calculate
 $\nabla_{kl} U_{ij} = \nabla_k \nabla_l U_{ij} - \Gamma_{kl}^m \nabla_m U_{ij}$, etc.

In the rest of this paper we shall always use orthonormal local frames.
Write equation~\eqref{3I-10} locally in the form
\begin{equation}
\label{3I-10'}
F (U) = \psi (x, u, \nabla u)
\end{equation}
where we identify $U \equiv \{U_{ij}\}$ and $F$ is the function defined by
\[ F (B) = f (\lambda (B)) \]
for a symmetric matrix $B$ with $\lambda (B) \in \Gamma$;
throughout the paper we shall use the notation
\[ F^{ij} = \frac{\partial F}{\partial B_{ij}} (U), \;\;
  F^{ij, kl} = \frac{\partial^2 F}{\partial B_{ij} \partial B_{kl}} (U). \]
The matrix $\{F^{ij}\}$ has eigenvalues $f_1, \ldots, f_n$ and
is positive definite by assumption (\ref{3I-20}), while (\ref{3I-30})
implies that $F$ is a concave function (see \cite{CNS3}).
Moreover, when $U_{ij}$ is diagonal so is $\{F^{ij}\}$, and the
following identities hold
\[   F^{ij} U_{ij} = \sum f_i \lambda_i, \;\;
F^{ij} U_{ik} U_{kj} = \sum f_i \lambda_i^2 \]
 where $\lambda (U) = (\lambda_1, \ldots, \lambda_n)$.

The following result can be found in \cite{Guan12a}
(Proposition 2.19 and Corollary 2.21).

\begin{proposition}[\cite{Guan12a}]
\label{g-prop1}
Suppose f satisfies \eqref{3I-20}.
There is $c_0 > 0$ and an index $r$ such that
\begin{equation}
\label{eq5.170}
\sum_{l<n} F^{ij} U_{il} U_{lj} \geq c_0 \sum_{i \neq r}  f_i \lambda_i^2.
\end{equation}

If in addition $f$ satisfies \eqref{3I-30} and \eqref{3I-45}  then
 for any index $r$,
\begin{equation}
\label{C210}
\sum f_i |\lambda_i| \leq \epsilon \sum_{i\neq r} f_i \lambda_i^2
  + C \Big(1 + \frac{1}{\epsilon} \sum f_i\Big).
\end{equation}
\end{proposition}

Let $\mathcal{L}$ be the linear operator locally defined by
\begin{equation}
\label{gj-lin}
\mathcal{L} v := F^{ij} \nabla_{ij} v
    + (F^{ij} A^{ij}_{p_k} - \psi_{p_k}) \nabla_k v, \;\; v \in C^2 (M)
\end{equation}
where $A^{ij}_{p_k} \equiv A^{ij}_{p_k} [u]
\equiv  A^{ij}_{p_k}  (x, u, \nabla u)$,
$\psi_{p_k} \equiv \psi_{p_k} [u] \equiv \psi_{p_k} (x, u, \nabla u)$.
The following is a consequence of Theorem~\ref{3I-th3}.

\begin{proposition}
\label{gj-lem3}
There exist uniform positive constants $R$, $\theta$
 such that
\begin{equation}
\label{gj-basic1}
\mathcal{L} (\ul u - u)
        \geq \theta \sum F^{ii} + \theta
\;\; \mbox{whenever $|\lambda (U)| \geq R$}.
\end{equation}
\end{proposition}

\begin{proof}
For any $x \in M$, choose a smooth orthonormal local frame
$e_1, \ldots, e_n$ about $x$ such that $\{U_{ij} (x)\}$ is diagonal.
From Lemma 6.2 in \cite{CNS3} and Theorem~\ref{3I-th3} we see
that there exist positive constants
$R$, $\theta$ such that when
$|\lambda (U)| \geq R$,
\begin{equation}
\label{gj-01}
F^{ii} (\ul U_{ii}  - U_{ii})
      \geq F (\ul U) -  F (U) + \theta \sum F^{ii} + \theta
\end{equation}
where $\ul U = \{\ul U_{ij}\} = \{\nabla_{ij} {\ul u} + A^{ij} [\ul u]\}$.
By (\ref{A2}) and \eqref{A4} we have
\[\begin{aligned}
A^{ii}_{p_k} \nabla_k (\ul u - u) \,& \geq
     A^{ii} (x, u, \nabla \ul u) - A^{ii} (x, u, \nabla u) \\
  \geq \,& A^{ii} (x, \ul u, \nabla \ul u) - A^{ii} (x, u, \nabla u)
\end{aligned}\]
and
\[\begin{aligned}
-\psi_{p_k} \nabla_k (\ul u - u) \,& \geq
     - \psi (x, u, \nabla \ul u) + \psi (x, u, \nabla u) \\
            \geq  \,& -\psi (x, \ul u, \nabla \ul u) + \psi (x, u, \nabla u).
\end{aligned}\]
Thus  \eqref{gj-basic1} follows from \eqref{gj-01}.
\end{proof}

\bigskip

\section{Proof of Theorem~\ref{3I-th3} }
\label{gj-C}
\setcounter{equation}{0}

\medskip
 
In this section we prove Theorem~\ref{3I-th3}. Throughout the section we
assume \eqref{3I-20}, \eqref{3I-30} and \eqref{gj-I100} hold. 
To give the reader some idea about the proof, we shall first 
prove the following simpler version of Theorem~\ref{3I-th3}.

\begin{theorem}
\label{3I-th3'}
Let $\mu \in \Gamma$ and 
$\sup_{\partial \Gamma}  f < \sigma < \sup_{\Gamma}  f$.
There exist positive constants
$\theta$, $R$ such that for any
$\lambda \in \partial \Gamma^{\sigma}$, when $|\lambda| \geq R$,
\begin{equation}
\label{3I-100a}
\sum f_i (\lambda) (\mu_{i} - \lambda_i) \geq
     \theta_{\mu} 
     + f (\mu) - f (\lambda).
\end{equation}
\end{theorem}

Recall that for $\sigma \in (\sup_{\partial \Gamma} f, \sup_{\Gamma} f)$ we have
$\Gamma^{\sigma} := \{f > \sigma\} \neq \emptyset$ and 
by assumptions \eqref{3I-20} and \eqref{3I-30}, 
$\partial \Gamma^{\sigma}$ is a smooth convex noncompact complete
hypersurface contained in $\Gamma$.
Let $\mu, \lambda \in \partial \Gamma^{\sigma}$. By the convexity of
$\partial \Gamma^{\sigma}$, the open segment
\[ (\mu, \lambda) \equiv \{t \mu + (1-t)\lambda: 0 < t < 1\} \]
is either completely contained in or does not intersect with
$\partial \Gamma^{\sigma}$. Therefore,
\[ f (t \mu + (1-t) \lambda) - \sigma > 0, \;\; \forall \, 0 < t < 1 \]
by condition~\eqref{3I-20}, unless
$(\mu, \lambda) \subset \partial \Gamma^{\sigma}$.

For $\lambda \in \Gamma$ we shall use $T_{\lambda}$ and $\nu_{\lambda}$
to denote the tangent plane and unit normal vector at 
$\lambda$ to $\partial \Gamma^{f(\lambda)}$, respectively. Note that 
$\nu_{\lambda} = Df (\lambda)/|Df (\lambda)|$.

\begin{proof}[Proof of Theorem~\ref{3I-th3'}]
We divide the proof into two cases: ({\bf a}) $f (\mu) \geq \sigma$
and ({\bf b}) $f (\mu) < \sigma$. 
For the first case we use ideas from \cite{Guan12a} where the
 case $f (\mu) = \sigma$ is proved. For case ({\bf b}) we introduce 
some new ideas which will be used in  the proof of Theorem~\ref{3I-th3}.

Case ({\bf a}) $f (\mu) \geq \sigma$. 
By assumption~\eqref{gj-I100} there is $R_0 > 0$ such that
$T_{\mu} \cap \partial \Gamma^{\sigma}$ is contained in the ball
$B_{R_0} (0)$.  By the convexity of
$\partial \Gamma^{\sigma}$, there exists $\delta > 0$ such that
for any $\lambda \in  \partial \Gamma^{\sigma}$ with
$|\lambda| \geq 2 R_0$,
the open segment from $\mu$ and $\lambda$
\[ (\mu, \lambda) \equiv \{t \mu + (1-t)\lambda: 0 < t < 1\} \]
intersects the level surface $\partial \Gamma^{f (\mu)}$ at a
unique point $\eta$ with $|\eta- \mu| > 2 \delta$.
Since $\nu_{\mu} \cdot (\eta- \mu)/|\eta- \mu|$ has a uniform 
positive lower bound (independent of $\lambda \in  \partial \Gamma^{\sigma}$ with
$|\lambda| \geq 2 R_0$) and the level hypersurface $\partial \Gamma^{f (\mu)}$
is smooth, the point $\mu + \delta|\eta- \mu|^{-1} (\eta - \mu)$ has a uniform 
positive distance  from $\partial \Gamma^{f (\mu)}$, and therefore
\[ f (\mu + \delta |\eta- \mu|^{-1} (\eta - \mu))
\geq f (\mu) + \theta_{\mu} \]
for some uniform constant $\theta > 0$. 
By the concavity of $f$,
\begin{equation}
\label{gj-C30}
\begin{aligned}
\sum f_i (\lambda) (\mu_i - \lambda_i)
\geq \,& \sum f_i (\lambda) (\eta_i - \lambda_i)
     + \sum f_i (\eta) (\mu_i - \eta_i) \\
\geq \,& f (\eta) - f (\lambda)
     + \sup_{0 \leq t \leq 1} f (t \mu + (1-t) \eta) - f (\mu) \\
\geq \,& f (\mu + \delta |\eta- \mu|^{-1} (\eta - \mu)) - f (\lambda) \\
\geq \,& \theta + f (\mu) - f (\lambda).
\end{aligned}
\end{equation}

We now assume $f (\mu) < \sigma$.
Assumption~\eqref{gj-I100} implies (see Lemma~\ref{gj-C-lemma10}
below) that there is
$\mu_0 \in \partial \Gamma^{\sigma}$ such that
$\mbox{dist} (\mu_0, T_{\mu}) = \mbox{dist} (\partial \Gamma^{\sigma}, T_{\mu})$
and therefore $T_{\mu_0}$ is parallel to $T_{\mu}$.
By assumption~\eqref{gj-I100} again there is $R_0 > 0$ such that
$T_{\mu_0} \cap \partial \Gamma^{f (\mu)}$ is contained in the ball
$B_{R_0} (0)$.
By the convexity of $\partial \Gamma^{\sigma}$ we have for any
$\lambda \in \partial \Gamma^{\sigma}$ with $|\lambda| \geq 2 R_0$,
\begin{equation}
\label{gj-C40}
\nu_{\mu_0} \cdot \nu_{\lambda}
      \leq \max_{\zeta \in \partial \Gamma^{\sigma}, \; |\zeta| = 2 R_0}
            \nu_{\mu_0} \cdot \nu_{\zeta}
       \equiv \beta < 1. 
\end{equation}
To see this one can consider the Gauss map
$G: \partial \Gamma^{\sigma} \rightarrow \bfS^n$;
the geodesic on $\bfS^n$ from $G (\mu_0) = \nu_{\mu_0}$ to 
$G ({\lambda}) = \nu_{\lambda}$ must intersect the image of
$\partial \Gamma^{\sigma} \cap \partial B_{R_0} (0)$.

Since $\partial \Gamma^{f (\mu)}$ is smooth, there exists $\delta > 0$ such that
\[ \mbox{dist} (\partial B_{\delta}^{\beta} (\mu), \partial \Gamma^{f (\mu)}) > 0 \]
where
$\partial B_{\delta}^{\beta} (\mu) = \{\zeta \in \partial B_{\delta} (\mu): \;
    \nu_{\mu} \cdot (\zeta - \mu)/\delta \geq \sqrt{1-\beta^2}\}$.
Therefore,
\begin{equation}
\label{gj-C50}
 \theta \equiv \inf_{\zeta \in \partial B_{\delta}^{\beta} (\mu)} f (\zeta) 
        - f(\mu) > 0. 
\end{equation}

For any $\lambda \in \partial \Gamma^{\sigma}$ with $|\lambda| \geq 2 R_0$,
let $P$ be the $2$-plane through $\mu$ spanned by $\nu_{\mu}$ and 
$\nu_{\lambda}$ (translated to $\mu$), and $L$ the line through $\mu$ and 
parallel to $P \cap T_{\lambda}$. 
From \eqref{gj-C40} we see that $L$ intersects
 $\partial B_{\delta}^{\beta} (\mu)$ at a unique point $\zeta$. 
By the concavity of $f$,
\begin{equation}
\label{gj-C60}
  \sum f_i (\lambda) (\mu_i - \lambda_i)
 = \sum f_i (\lambda) (\zeta_i - \lambda_i)
  \geq f (\zeta) - f (\lambda) \geq \theta + f(\mu) - f (\lambda). 
\end{equation}
This completes the proof for case ({\bf b}). 
\end{proof}

A careful examination of the above proof for case ({\bf b}) shows that it
actually works for case ({\bf a}) too (with some obvious modifications).
In what follows we shall extend the ideas to give a proof of 
Theorem~\ref{3I-th3}. 

We need the following lemma from \cite{Guan12a}.

\begin{lemma}[\cite{Guan12a}]
\label{3I-C-lemma2}
Let $\mu \in \partial \Gamma^{\sigma}$.
Then for any $t > 0$, the part of $\partial \Gamma^{\sigma}$
\[ \Sigma_t = \Sigma_t  (\mu) := \{\lambda \in \partial \Gamma^{\sigma}:
       (\lambda - \mu) \cdot \nu_{\mu} \leq t\} \]
is a convex cap with smooth compact boundary on the plane
$t \nu_{\mu} + T_{\mu} \partial \Gamma^{\sigma}$.
\end{lemma}

\begin{lemma}
\label{gj-C-lemma10}
The set 
\[  L^{\sigma} (\mu) = \Big\{\zeta \in \partial \Gamma^{\sigma}:
\nu_{\mu} \cdot (\zeta - \mu) 
    = \min_{\lambda \in \partial \Gamma^{\sigma}}
                              \nu_{\mu} \cdot (\lambda - \mu)\Big\} \]
is nonempty and compact. Moreover, $\nu_{\lambda} = \nu_{\mu}$
for all $\lambda \in L^{\sigma} (\mu)$. 
\end{lemma}

\begin{proof}
If $f (\mu) < \sigma$ this follows from Lemma~\ref{3I-C-lemma2}
while if $f (\mu) > \sigma$ we see this directly from 
assumption~\eqref{gj-I100}.
\end{proof}

\begin{lemma}
\label{gj-C-lemma20}
If $R' \geq R > \max \{|\lambda|: 
     \lambda \in L^{\sigma} (\mu)\}$ 
then $\beta_{R'} (\mu, \sigma) \leq \beta_R (\mu, \sigma) < 1$.
\end{lemma}

\begin{proof}
We first note that 
 $\nu_{\mu} = \nu_{\mu^{\sigma}}$ and therefore 
$\beta_R (\mu, \sigma) = \beta_R (\mu^{\sigma}, \sigma)$
for $\mu^{\sigma} \in L^{\sigma} (\mu)$. 
For any $\lambda \in \partial \Gamma^{\sigma} \setminus L^{\sigma} (\mu)$ we have $0 < \nu_{\mu} \cdot \nu_{\lambda} < 1$
since $\nu_{\mu}, \nu_{\lambda} \in \Gamma_n$ and 
$\nu_{\mu} \neq \nu_{\lambda}$ 
by the convexity of $\partial \Gamma^{\sigma}$. Therefore
$\beta_R < 1$ for 
$R > \max \{|\mu^{\sigma}|: \mu^{\sigma} \in L^{\sigma} (\mu)\}$
by compactness (of
$\partial \Gamma^{\sigma} \cap \partial B_R$).
To see that $\beta_R$ is non-increasing in $R$
we consider the Gauss map
$G: \partial \Gamma^{\sigma} \rightarrow \bfS^n$. 
Suppose $\lambda \in \partial \Gamma^{\sigma}$
and $|\lambda| > R > \max \{|\mu^{\sigma}|: \mu^{\sigma} \in L^{\sigma} (\mu)\}$. Then 
the geodesic on $\bfS^n$ from $G (L^{\sigma} (\mu))$ to $G ({\lambda})$
must intersect the image of
$\partial \Gamma^{\sigma} \cap \partial B_{R} (0)$, that is, it
contains a point $G (\lambda_R)$ for some 
$\lambda_R \in \partial \Gamma^{\sigma} \cap \partial B_{R} (0)$. 
Consequently, 
\[ \nu_{\mu} \cdot \nu_{\lambda} 
     \leq \nu_{\mu} \cdot \nu_{\lambda_R} \leq \beta_R. \]
So Lemma~\ref{gj-C-lemma20} holds.
\end{proof}

\begin{lemma}
\label{gj-C-lemma30}
Let $K$ be a compact subset of $\Gamma$ 
and $\sup_{\partial \Gamma}  f < a \leq b < \sup_{\Gamma}  f$. 
Then
\begin{equation}
\label{gj-C100}
  R_0 (K, [a, b]) \equiv
  \sup_{\mu \in K} \sup_{a \leq \sigma \leq b} \max_{\lambda \in L^{\sigma} (\mu)} 
    |\lambda| < \infty 
\end{equation}
and 
\begin{equation}
\label{gj-C110}
0 < \beta_{R'} (K, [a, b]) \leq \beta_R (K, [a, b])  < 1, 
       \;\; \forall \,R' \geq R > R_0 (K, [a, b]), 
\end{equation}
where 
\[ \beta_R (K, [a, b]) \equiv \sup_{\mu \in K} \beta_R (\mu, [a, b])
   \equiv \sup_{\mu \in K} \sup_{a \leq \sigma \leq b} \beta_R (\mu, \sigma). \]
\end{lemma}

\begin{proof}
We first show that for any $\mu \in \Gamma$, 
\begin{equation}
\label{gj-C120}
 R_0 (\mu, [a, b]) \equiv
  \sup_{a \leq \sigma \leq b} \max_{\lambda \in L^{\sigma} (\mu)} 
    |\lambda| < \infty. 
\end{equation}
We consider two cases: ({\bf i}) $f (\mu) \leq a$ and 
({\bf ii}) $f (\mu) \geq b$; the case $a \leq f (\mu) \leq b$ follows 
obviously. 
In case ({\bf i}) the set 
\[ L^{[a,b]} (\mu) \equiv \bigcup_{a \leq \sigma \leq b}  L^{\sigma} (\mu) \]
is contained in the region bounded by $\Sigma_t (\mu)$ and $T_{\mu^b}$, 
where $t = \nu_{\mu} \cdot (\mu^b - \mu)$ for any $\mu^b \in L^b (\mu)$ which 
is nonempty by Lemma~\ref{gj-C-lemma10},
while in case ({\bf ii}) it is in the (bounded) subregion of $\Gamma^b$
cut by $T_{\mu}$. So in both cases \eqref{gj-C120} holds.

Next, we show by contradiction that 
\begin{equation}
\label{gj-C130}
\beta_R (\mu, [a, b]) \equiv 
       \sup_{a \leq \sigma \leq b} \beta_R (\mu, \sigma)  < 1,  
       \;\; \forall \, R > R_0 (\mu, [a, b]). 
\end{equation}
Suppose for each integer $k \geq 1$
there exists $\lambda_k \in \Gamma$ with $|\lambda_k| = R$
and $a \leq f (\lambda_k) \leq b$ such that 
\[ \nu_{\mu} \cdot \nu_{\lambda_k} \geq 1 - \frac{1}{k}. \]
Then by compactness we obtain a point $\lambda \in \Gamma$
with $a \leq f (\lambda) \leq b$, $\nu_{\mu} \cdot \nu_{\lambda} = 1$ 
and $|\lambda| = R > R_0 (\mu, f (\lambda))$. 
This contradicts Lemma~\ref{gj-C-lemma20}
from which we also see that $\beta_R (\mu, [a, b])$ is 
nonincreasing in $R$ for $R > R_0 (\mu, [a, b])$. 

Suppose now that for each integer $k \geq 1$
there exists $\mu_k \in K$, $\sigma_k \in [a, b]$ and 
$\lambda_k \in  L^{\sigma_k} (\mu_k)$ with $|\lambda_k| \geq k$. 
By the compactness of $K$ 
we may assume $\mu_k \rightarrow \mu \in K$ as $k \rightarrow \infty$.
Since $\nu_{\mu_k} \cdot \nu_{\lambda_k} = 1$, and 
$\partial \Gamma^{\sigma_k}$ is smooth
we have
\[ \lim_{k \rightarrow \infty} \nu_{\mu} \cdot \nu_{\lambda_k} 
   = 1 + \lim_{k \rightarrow \infty} (\nu_{\mu} - \nu_{\mu_k}) \cdot \nu_{\lambda_k} 
   = 1. \]
This contradicts \eqref{gj-C130} and the monotonicity of 
$\beta_R (\mu, [a, b])$, proving \eqref{gj-C100}. 
The proof of \eqref{gj-C110} is now obvious.
\end{proof}

We are now ready to prove Theorem~\ref{3I-th3}. 

\begin{proof}[Proof of Theorem~\ref{3I-th3}]
The proof below is a straightforward modification of the second part of 
the proof of Theorem~\ref{3I-th3'}; we include it here for
completeness and the reader's convenience. 

Let $\epsilon = \frac{1}{2} \mbox{dist} (K, \partial \Gamma)$ and 
\[ K^{\epsilon} = \{\mu - t {\bf 1}: \mu \in K, \; 0 \leq t \leq \epsilon\} \]
where ${\bf 1} = (1, \ldots, 1) \in \bfR^n$. Then $K^{\epsilon}$ is 
a compact subset of $\Gamma$. Let $R = 2 R_0 (K^{\epsilon}, [a, b])$,
and $\beta = \beta_R (K^{\epsilon}, [a, b])$. 
By Lemma~\ref{gj-C-lemma30}, $0 < \beta < 1$ and therefore,
since $f$ is smooth and $Df \neq 0$ everywhere, by the compactness of 
$K^{\epsilon}$ there exists 
$\delta > 0$ depending on $\beta$ and bounds on the (principal) 
curvatures of $\partial \Gamma^{f (\mu)}$ for all $\mu \in K^{\epsilon}$,
such that
\[ \inf_{\mu \in K^{\epsilon}} 
 \mbox{dist} (\partial B_{\delta}^{\beta} (\mu), \partial \Gamma^{f (\mu)}) > 0 \]
where $\partial B_{\delta}^{\beta} (\mu)$ denotes the spherical cap as
in the proof of Theorem~\ref{3I-th3'}.
Consequently,
\begin{equation}
\label{gj-C150}
 \theta \equiv \frac{1}{2} \inf_{\mu \in K^{\epsilon}} 
  \inf_{\zeta \in \partial B_{\delta}^{\beta} (\mu)}  (f (\zeta) - f(\mu)) > 0. 
\end{equation}

Next, for any 
$\lambda \in \Gamma^{[a,b]} = \ol{\Gamma^a} \setminus \Gamma^b$ 
with $|\lambda| \geq R$
and ${\mu} \in K^{\epsilon}$, as in the proof of Theorem~\ref{3I-th3'}
let $P$ be the $2$-plane through $\mu$ spanned by $\nu_{\mu}$ and 
$\nu_{\lambda}$ (translated to $\mu$), and $L$ the line through $\mu$ and 
parallel to $P \cap T_{\lambda}$. 
Since $\nu_{\mu} \cdot \nu_{\lambda} \leq \beta < 1$ by 
Lemma~\ref{gj-C-lemma30},
$L$ intersects $\partial B_{\delta}^{\beta} (\mu)$ at a unique point $\zeta$, 
and therefore by the concavity of $f$, 
\begin{equation}
\label{gj-C160}
  \sum f_i (\lambda) (\mu_i - \lambda_i)
 = \sum f_i (\lambda) (\zeta_i - \lambda_i)
  \geq f (\zeta) - f (\lambda) \geq 2 \theta + f(\mu) - f (\lambda). 
\end{equation}

Finally, since $K$ is compact, by the continuity of $f$ there exists 
$\epsilon_0 \in (0, \epsilon]$ such that 
\[ f (\mu - t {\bf 1}) \geq f (\mu) - \theta, 
   \;\; \forall \, \mu \in K, \; 0 \leq t \leq \epsilon_0. \]
Combining this with \eqref{gj-C160} gives \eqref{3I-100a'}
for $\theta (K, [a,b]) = \min \{\epsilon_0, \theta\}$. 
\end{proof}

\bigskip

\section{Interior and global estimates for second derivatives}
\label{gj-S}
\setcounter{equation}{0}

\medskip

In this section we derive
the estimates for second derivatives in Theorem~\ref{gj-th2} and
Theorem~\ref{gj-th4}.

Let
\[ W (x) = \max_{\xi\in T_{x}M,|\xi|=1}
     (\nabla_{\xi\xi} u + A^{\xi\xi} (x, u,\nabla u)) e^{\phi} \]
where $\phi$ is a function to be determined,
and assume
\[ W (x_0) = \max_{\bM} W \]
for an interior point $x_{0} \in M$.
Choose a smooth orthonormal local frame $e_{1}, \ldots, e_{n}$ about $x_{0}$ such that $\nabla_{e_i} e_j = 0$ and
$U_{ij}$ is diagonal at $x_0$.
We assume
\[ U_{11} (x_0) \geq \cdots \geq U_{nn} (x_0) \]
so
$W (x_0) = U_{11} (x_0) e^{\phi (x_0)}$.

The function
$\log U_{11} + \phi$ attains its maximum at $x_{0}$ where
\begin{equation}
\label{gj-S20}
\frac{\nabla_i U_{11}}{U_{11}} + \nabla_i \phi = 0,
\end{equation}
\begin{equation}
\label{gj-S30}
\begin{aligned}
\frac{\nabla_{ii} U_{11}}{U_{11}}
   - \Big(\frac{\nabla_i U_{11}}{U_{11}}\Big)^2 + \nabla_{ii} \phi \leq 0.
\end{aligned}
\end{equation}
By simple calculation,
\begin{equation}
\label{gj-S40}
(\nabla_i U_{11})^2 \leq (\nabla_1 U_{1i})^2 + C U_{11}^2,
\end{equation}
\begin{equation}
\label{gj-S50}
\nabla_{ii} U_{11}
  \geq \nabla_{11} U_{ii} + \nabla_{ii} A^{11} - \nabla_{11} A^{ii} - C U_{11}.
\end{equation}

Differentiating equation (\ref{3I-10'}) twice, we obtain at $x_{0}$,
\begin{equation}
\label{gj-S60}
F^{ii} \nabla_{k} U_{ii} = \nabla_k \psi + \psi_{u} \nabla_{k} u
            + \psi_{p_j} \nabla_{kj} u,
\end{equation}
and, by \eqref{gj-S20},
\begin{equation}
\label{gj-S70}
\begin{aligned}
F^{ii} \nabla_{11} U_{ii}
     + \,& F^{ij,kl} \nabla_{1} U_{ij} \nabla_{1} U_{kl} \\
  \geq \,& \psi_{p_{j}} \nabla_{11j} u + \psi_{p_{l} p_{k}} \nabla_{1k} u \nabla_{1l} u
       - C U_{11} \\
  \geq \,& \psi_{p_{j}} \nabla_j U_{11} + \psi_{p_{1} p_{1}} U_{11}^2 - C U_{11} \\
    = \,& - U_{11} \psi_{p_{j}} \nabla_j \phi + \psi_{p_{1} p_{1}} U_{11}^2 - C U_{11}.
\end{aligned}
\end{equation}
Next,
\begin{equation}
\label{gj-S80}
\begin{aligned}
F^{ii}(\nabla_{ii} A^{11} - \nabla_{11} A^{ii})
 \geq \,& F^{ii} (A^{11}_{p_{j}} \nabla_{iij} u - A^{ii}_{p_{j}} \nabla_{11j} u) \\
      \,& + F^{ii} (A^{11}_{p_i p_i} U_{ii}^2 - A^{ii}_{p_1 p_1} U_{11}^2)
          - C U_{11} \sum F^{ii} \\
 \geq \,& U_{11} F^{ii} A^{ii}_{p_{j}} \nabla_{j} \phi
          - C U_{11} \sum F^{ii} - C U_{11} \\
        & - C \sum_{i \geq 2} F^{ii} U_{ii}^2
          - U_{11}^2 \sum_{i \geq 2} F^{ii} A^{ii}_{p_1 p_1}.
\end{aligned}
\end{equation}

By \eqref{gj-S30}, \eqref{gj-S50},
\eqref{gj-S70} and \eqref{gj-S80}
we obtain
\begin{equation}
\label{gj-S90}
\begin{aligned}
\mathcal{L} \phi
  \leq \,& U_{11} \sum_{i \geq 2} F^{ii} A^{ii}_{p_1 p_1} - \psi_{p_{1} p_{1}} U_{11}
           + E + \frac{C}{U_{11}} F^{ii} U_{ii}^2 + C \sum F^{ii} + C
\end{aligned}
\end{equation}
where
\[ E = \frac{1}{U_{11}^2} F^{ii} (\nabla_i U_{11})^2
            + \frac{1}{U_{11}} F^{ij,kl} \nabla_{1} U_{ij} \nabla_{1} U_{kl}. \]

 Let
 \[ \phi = \frac{\delta |\nabla u|^2}{2} + b \eta \]
where $b$, $\delta$ are undetermined constants, $0 < \delta < 1 \leq b$, and $\eta$ is a $C^{2}$
function which may depend on $u$ but not on its derivatives.
We have
\begin{equation}
\label{gs3}
\begin{aligned}
\nabla_{i} \phi
     = \,& \delta \nabla_{j} u \nabla_{ij} u + b \nabla_{i} \eta
     = \delta \nabla_i u U_{ii} - \delta \nabla_{j} u A^{ij} + b \nabla_{i} \eta
\end{aligned}
\end{equation}
\begin{equation}
\label{gs4}
\begin{aligned}
\nabla_{ii} \phi
     = \,& \delta  \nabla_{ij} u \nabla_{ij} u + \delta \nabla_{j} u \nabla_{iij} u
            + b \nabla_{ii} \eta  \\
   \geq \,& \frac{\delta}{2} U_{ii}^2 - C \delta + \delta \nabla_{j} u \nabla_{iij} u
            + b \nabla_{ii} \eta.
\end{aligned}
\end{equation}
By (\ref{hess-A70}), \eqref{gj-P10}, 
and (\ref{gj-S60}) we see that
\begin{equation}
\label{gs7}
\begin{aligned}
F^{ii} \nabla_{j} u \nabla_{iij} u
  \geq \,& F^{ii} \nabla_{j} u (\nabla_{j} U_{ii} - \nabla_{j} A^{ii})
            - C |\nabla u|^2 \sum F^{ii} \\
  \geq \,& (\psi_{p_{l}} - F^{ii} A^{ii}_{p_l}) \nabla_{j} u \nabla_{jl} u
           - C |\nabla u|^2 \sum F^{ii}- C |\nabla u|^2.
\end{aligned}
\end{equation}
It follows that
\begin{equation}
\label{gj-S100}
\begin{aligned}
\mathcal{L} \phi \geq b \mathcal{L} \eta + \frac{\delta}{2} F^{ii} U_{ii}^2
   - C \sum F^{ii} - C.
\end{aligned}
\end{equation}

We now go on to prove \eqref{hess-a10} in Theorem~\ref{gj-th2}.
Let $\eta = \ul u - u$ so by \eqref{gs3},
\begin{equation}
\label{gs3.5}
\begin{aligned}
(\nabla_{i} \phi)^2
  \leq C \delta^2 (1 + U_{ii}^2) + 2 b^2 (\nabla_{i} \eta)^2
  \leq C \delta^2 U_{ii}^2 + C b^2.
\end{aligned}
\end{equation}
Next we estimate $E$. For fixed $0 < s \leq 1/3$ let
\[  \begin{aligned}
J \,& = \{i: U_{ii} \leq - s U_{11}\}, \;\;
K = \{i:  U_{ii} > - s U_{11} \}.
  \end{aligned} \]
We have
\begin{equation}
\label{gj-S130}
\begin{aligned}
 - F^{ij, kl} \nabla_1 U_{ij} \nabla_1 U_{kl}
\geq \,& \sum_{i \neq j} \frac{F^{ii} - F^{jj}}{U_{jj} - U_{ii}}
           (\nabla_1 U_{ij})^2 \\
 \geq \,& 2 \sum_{i \geq 2} \frac{F^{ii} - F^{11}}{U_{11} - U_{ii}}
            (\nabla_1 U_{i1})^2 \\
 \geq \,& \frac{2}{(1+s) U_{11}} \sum_{i \in K} (F^{ii} - F^{11})
            (\nabla_1 U_{i1})^2 \\
\geq \,& \frac{2 (1-s)}{(1+s) U_{11}} \sum_{i \in K} (F^{ii} - F^{11})
            ((\nabla_i U_{11})^2 - C U_{11}^2/s).
\end{aligned}
\end{equation}
The first inequality in \eqref{gj-S130} is a consequence of an
inequality due to Andrews~\cite{Andrews94}
and Gerhardt~\cite{Gerhardt96}; it was also included in an original version of \cite{CNS3}.
By \eqref{gj-S130}, \eqref{gj-S20} and \eqref{gs3.5} we obtain
\begin{equation}
\label{gj-S140}
\begin{aligned}
 E \leq \,& \frac{1}{U_{11}^2} \sum_{i \in J} F^{ii} (\nabla_i U_{11})^2
             + C \sum_{i \in K} F^{ii}
             + \frac{C F^{11}}{U_{11}^2}  \sum_{i \in K} (\nabla_i U_{11})^2  \\
   \leq \,& \sum_{i \in J} F^{ii} (\nabla_i \phi)^2
             +  C \sum F^{ii} + C F^{11} \sum (\nabla_i \phi)^2 \\
   \leq \,& C b^2 \sum_{i \in J} F^{ii}  + C \delta^2 \sum F^{ii} U_{ii}^2
            +  C \sum F^{ii} + C (\delta^2 U_{11}^2 + b^2) F^{11}.
\end{aligned}
\end{equation}

It follows from \eqref{gj-S90}, \eqref{gj-S100} and \eqref{gj-S140}
that
\begin{equation}
\label{gj-S150}
\begin{aligned}
b \mathcal{L} \eta
   \leq \,& \Big(C \delta^2 - \frac{\delta}{2} + \frac{C}{U_{11}}\Big) F^{ii} U_{ii}^2
            + C b^2 \sum_{i \in J} F^{ii} + C \sum F^{ii} \\
          & + C (\delta^2 U_{11}^2 + b^2) F^{11} + C.
\end{aligned}
\end{equation}
By Proposition~\ref{gj-lem3}
there exist uniform positive constants $\theta$, $R$ satisfying
\[\mathcal{L} (\ul u - u) \geq \theta \sum F^{ii} + \theta\]
provided that $U_{11} (x_0) > R$ and hence, when $b$ is sufficiently large,
\[\Big(C \delta^2 - \frac{\delta}{2} + \frac{C}{U_{11}}\Big) F^{ii} U_{ii}^2
            + C b^2 \sum_{i \in J} F^{ii} + C (\delta^2 U_{11}^2 + b^2) F^{11}
\geq 0. \]
This implies a bound $U_{11} (x_0) \leq C$ as otherwise the
first term would be negative for $\delta$ chosen sufficiently small, and
$|U_{ii}| \geq s U_{11}$ for $i \in J$.
 The proof of \eqref{hess-a10} in Theorem~\ref{gj-th2} is therefore
 complete.


We now turn to the interior estimate \eqref{gj-G185}
in Theorem~\ref{gj-th4}.
Following \cite{GW03a} we choose a cutoff function
$\zeta\in C_{0}^{\infty}(B_{r})$
such that
\begin{equation}
\label{2-22}
0\leq\zeta\leq 1, ~~\zeta|_{B_{\frac{r}{2}}}\equiv 1,
~~|\nabla \zeta| \leq \frac{C_r}{\sqrt{\zeta}},
~~|\nabla^{2} \zeta| \leq C_r
\end{equation}
where $C_r$ is a constant depending on $r$.

Let $\eta = \ul u - u + \log \zeta$.
By \eqref{gj-S20} and \eqref{gs3.5},
\begin{equation}
\label{gj-S105}
\begin{aligned}
E \leq F^{ii} (\nabla_i \phi)^2
   \leq \,& C \delta^2 F^{ii} U_{ii}^2 +
   \frac{C b^2}{\zeta^2} F^{ii} (\nabla_i \zeta)^2 + C b^2 \sum F^{ii} \\
  \leq \,& C \delta^2 F^{ii} U_{ii}^2 + \frac{C b^2}{\zeta} \sum F^{ii}.
\end{aligned}
\end{equation}

Under the MTW condition~\eqref{A3} we have
\[ A^{ii}_{p_1 p_1} \leq - c_0 < 0. \]
It therefore follows from \eqref{gj-S90},  \eqref{gj-S100} and
\eqref{gj-S105} that
\begin{equation}
\label{gj-S110}
\begin{aligned}
 b \mathcal{L} \eta
   \leq \,& \Big(\frac{C}{U_{11}} + C \delta^2- \frac{\delta}{2}\Big) F^{ii} U_{ii}^2
    - c_0 U_{11}  \sum_{i \geq 2} F^{ii} \\
    & + C b^2 \Big(\frac{1}{\zeta} + \frac{1}{U_{11}}\Big) \sum F^{ii} + C.
\end{aligned}
\end{equation}

From the assumption $\ul u \in C^2 (\ol B_r)$ is admissible so
for $B > 0$ sufficiently large,
\[ \lambda (B g + \ul U) \in \Gamma_n \;\; \mbox{in $\ol B_r$} \]
and therefore,
\[ F (2 B g + \ul U) \geq F (B g)   \;\; \mbox{in $\ol B_r$}. \]
By the concavity of $F$,
\begin{equation}
\label{g12}
\begin{aligned}
F^{ii} (\ul{U}_{ii}- U_{ii}) 
 & \geq F (2 B g + \ul U) - F (U) - 2 B \sum F^{ii}\\
& \geq F (Bg) - 2 B \sum F^{ii} - \psi (x, u, \nabla u).
\end{aligned}
\end{equation}

For $B > 0$ sufficiently large we have
\begin{equation}
\label{gj-S120}
\begin{aligned}
 \mathcal{L} \eta
  \geq \,& \mathcal{L} (\ul u - u) - \frac{C}{\zeta} \sum F^{ii} \\
  \geq \,& F^{ii} (\ul U_{ii} - U_{ii}) - \frac{C}{\zeta} \sum F^{ii} - C \\
  \geq  \,& F (B g) - \Big(2 B + \frac{C}{\zeta}\Big) \sum F^{ii} - C.
\end{aligned}
\end{equation}
From  \eqref{gj-S110}, \eqref{gj-S120} we derive a bound
$\zeta (x_0) U_{11} (x_0) \leq C$ which yields
$W (x_0) \leq C$ when we fix $\delta$ small and $B$ large.
  This proves
\[ |\nabla^2 u| \leq \frac{C}{\zeta} \;\; \mbox{in $B_r$}. \]
The proof of \eqref{gj-G185} in Theorem~\ref{gj-th4} is complete.

\bigskip

\section{Boundary estimates for second derivatives}
\label{gj-B}
\setcounter{equation}{0}

\medskip

In this section we establish the boundary estimate \eqref{hess-a10b}
in Theorem~\ref{gj-th2}.
We shall assume that the function $\varphi \in C^4 (\partial M)$ is
extended to a $C^4$ function on $\bM$, still denoted $\varphi$.

For a point $x_0$ on $\partial M$, we shall choose
smooth orthonormal local frames $e_1, \ldots, e_n$ around $x_0$ such that
when restricted to $\partial M$, $e_n$ is normal to $\partial M$.
For $x \in \bM$ let $\rho (x)$ and $d (x)$ denote the distances
from $x$ to $x_0$ and $\partial M$, respectively,
\[ \rho (x) \equiv \mbox{dist}_{M^n} (x, x_0), \;\;
      d (x) \equiv \mbox{dist}_{M^n} (x, \partial M) \]
and $M_{\delta} = \{x \in M : \rho (x) < \delta \}$.

Since $u - \ul{u} = 0$ on $\partial M$ we have
\begin{equation}
\label{hess-a200}
\nabla_{\alpha \beta} (u - \ul{u})
 = -  \nabla_n (u - \ul{u}) \varPi (e_{\alpha}, e_{\beta}), \;\;
\forall \; 1 \leq \alpha, \beta < n \;\;
\mbox{on  $\partial M$}
\end{equation}
where 
$\varPi$ denotes the second fundamental form of $\partial M$.
Therefore,
\begin{equation}
\label{hess-E130}
|\nabla_{\alpha \beta} u| \leq  C,  \;\; \forall \; 1 \leq \alpha, \beta < n
\;\;\mbox{on} \;\; \partial M.
\end{equation}


To proceed we calculate using \eqref{gj-S60} and \eqref{hess-A70}
for each $1 \leq k \leq n$,
\begin{equation}
\label{hess-E170}
\begin{aligned}
|\mathcal{L} \nabla_k (u - \varphi)|
\leq \,& C \Big(1 + \sum f_i |\lambda_i| + \sum f_i\Big),
\end{aligned}
\end{equation}
and
\begin{equation}
\begin{aligned}
\mathcal{L} |\nabla_k (u - \varphi)|^2
\geq \,& F^{ij} U_{i \beta} U_{j \beta}
         - C \Big(1 + \sum f_i |\lambda_i| + \sum f_i\Big).
\end{aligned}
\end{equation}

As in \cite{Guan12a} we need the following crucial lemma.

\begin{lemma}
\label{ma-lemma-E10}
There exist some uniform positive constants
$t, \delta, \varepsilon$ sufficiently small and $N$ sufficiently large
such that the function
\begin{equation}
\label{ma-E85}
v = (u - \ul{u}) + t d - \frac{N d^2}{2}
\end{equation}
satisfies $v \geq 0$ on $\bar{M_{\delta}}$ and
\begin{equation}
\label{ma-E86}
\mathcal{L} v 
\leq - \varepsilon \Big(1 + \sum F^{ii}\Big)
   \;\; \mbox{in $M_{\delta}$}.
\end{equation}
\end{lemma}

The proof of Lemma~\ref{ma-lemma-E10} is similar to that of Lemma 4.1
in \cite{Guan12a} using Proposition~\ref{gj-lem3},
so we omit it here.

Let
\begin{equation}
\label{hess-E176}
 \varPsi
   = A_1 v + A_2 \rho^2 - A_3 \sum_{\beta < n} |\nabla_{\beta} (u - \varphi)|^2.
\end{equation}
For fixed $ 1 \leq \alpha < n$, we derive using
Lemma~\ref{ma-lemma-E10} and Proposition~\ref{g-prop1} as in \cite{Guan12a}
\begin{equation}
\label{hess-E170'}
  \left\{ \begin{aligned}
& \mathcal{L} (\varPsi \pm \nabla_{\alpha} (u - \varphi)) \leq 0 \;\;
 \mbox{in $M_{\delta}$}, \\
 & \varPsi \pm \nabla_{\alpha} (u - \varphi) \geq 0
\;\; \mbox{ on $\partial M_{\delta}$}
\end{aligned} \right.
\end{equation}
when $A_1 \gg A_2 \gg A_3 \gg 1$. By the maximum principle we derive
$\varPsi \pm \nabla_{\alpha} (u - \varphi) \geq 0$
in $M_{\delta}$ and therefore
\begin{equation}
\label{hess-E130'}
|\nabla_{n\alpha} u (x_0)| \leq \nabla_n \varPsi (x_0) \leq C,
\;\; \forall \; \alpha < n.
\end{equation}

The rest of this section is devoted to derive
\begin{equation}
\label{cma-200}
 \nabla_{n n} u (x_0) \leq C.
\end{equation}
The idea is similar to that used in \cite{Guan12a} but the proof is much
more complicated due to the dependence of $\psi$ on $u$ and $\nabla u$.
So we shall carry out the proof in detail.

As in \cite{Guan12a}, following an idea of Trudinger~\cite{Trudinger95}
we prove that there are uniform constants $c_0, R_0$
such that for all $R > R_0$,
$(\lambda' [\{U_{\alpha \beta} (x_0)\}], R) \in \Gamma$ and
\begin{equation}
\label{cma-201}
f (\lambda' [\{U_{\alpha \beta} (x_0)\}], R) \geq \psi [u] (x_0) + c_0
\end{equation}
which implies \eqref{cma-200} by Lemma 1.2 in \cite{CNS3},
where
$\lambda' [\{U_{\alpha \beta}\}] = (\lambda'_1, \cdots, \lambda'_{n-1})$
denotes the eigenvalues of the $(n-1) \times (n-1)$ matrix
$\{U_{\alpha \beta}\}$ ($1 \leq \alpha, \beta \leq n-1$),
and $\psi [u] = \psi (\cdot, u, \nabla u)$.

Let 
\[ \tilde{m} \equiv \min_{x_0 \in \partial M} 
    (\lim_{R \rightarrow + \infty} f (\lambda' [\{U_{\alpha \beta} (x_0)\}], R)
       - \psi (x_0)), \]
\[  \tilde{c} \equiv \min_{x_0 \in \partial M} 
     (\lim_{R \rightarrow + \infty} 
         f (\lambda' [\{\ul U_{\alpha \beta} (x_0) \}], R)
        - F (\ul U_{ij} (x_0))) > 0. \]
We wish to show $\tilde{m} > 0$. Without loss of generality
we assume $\tilde{m} < \tilde{c}/2$ (otherwise we are done) and
suppose $\tilde{m}$ is achieved at a point $x_0 \in \partial M$.
Choose local orthonormal frames around $x_0$ as before and assume
$\nabla_{n n} u (x_0) \geq \nabla_{n n} \ul u (x_0)$.

For a symmetric $(n-1)^2$ matrix $\{r_{\alpha  {\beta}}\}$ such that $(\lambda' [\{r_{\alpha \beta}\}], R) \in \Gamma$ when $R$ is 
sufficientluy large, define
\[ \tF [r_{\alpha \beta}]
   \equiv \lim_{R \rightarrow + \infty} f (\lambda' [\{r_{\alpha \beta}\}], R) \]
Note that  $\tF$ is concave by \eqref{3I-30}.
There exists a positive semidefinite matrix $\{\tF^{\alpha {\beta}}_0\}$
such that 

\begin{equation}
\label{c-200}
\tF^{\alpha {\beta}}_0 (r_{\alpha  {\beta}} - U_{\alpha {\beta}} (x_0))
    \geq \tF [r_{\alpha  {\beta}}] - \tF [U_{\alpha  {\beta}} (x_0)]
 \end{equation}
for any symmetric matrix $\{r_{\alpha  \beta}\}$ with
$(\lambda' [\{r_{\alpha  \beta}\}], R) \in \Gamma$ when $R$ is 
sufficiently large.
In particular,
\begin{equation}
\label{c-210}
\tF^{\alpha {\beta}}_0 U_{\alpha  {\beta}} - \psi [u]
- \tF^{\alpha {\beta}}_0 U_{\alpha  {\beta}} (x_0) + \psi [u] (x_0)
\geq \tF [U_{\alpha  {\beta}}] - \psi [u] - \tilde{m} \geq 0 \;\;
\mbox{on $\partial M$}.
\end{equation}

By \eqref{hess-a200} we have on $\partial M$,
\begin{equation}
\label{c-220}
 U_{\alpha {\beta}} = \ul{U}_{\alpha {\beta}}
    - \nabla_n (u - \ul{u}) \sigma_{\alpha {\beta}}
+ A^{\alpha \beta}[u] - A^{\alpha \beta}[\ul u]
\end{equation}
where
$\sigma_{\alpha {\beta}} = \langle \nabla_{\alpha} e_{\beta}, e_n \rangle$;
note that
$\sigma_{\alpha \beta} = \varPi (e_\alpha, e_\beta)$ on
$\partial M$. 
It follows that at $x_0$,
\begin{equation}
\label{c-225}
\begin{aligned}
 \nabla_n (u - \ul{u}) \tF^{\alpha {\beta}}_0 \sigma_{\alpha {\beta}}
   = \,& \tF^{\alpha {\beta}}_0 (\ul{U}_{\alpha \beta} - U_{\alpha  {\beta}})
          + \tF^{\alpha {\beta}}_0 (A^{\alpha \beta}[u] - A^{\alpha \beta}[\ul u]) \\
\geq \,& \tF[\ul{U}_{\alpha {\beta}}] - \tF[U_{\alpha {\beta}}]
          + \tF^{\alpha {\beta}}_0 (A^{\alpha \beta}[u] - A^{\alpha \beta}[\ul u]) \\
  =  \,& \tF[\ul{U}_{\alpha {\beta}}] - \psi [u] - \tilde{m}
          + \tF^{\alpha {\beta}}_0 (A^{\alpha \beta}[u] - A^{\alpha \beta}[\ul u]) \\
 \geq \,& \tilde{c} - \tilde{m} + \psi [\ul u] - \psi [u]
         + \tF^{\alpha {\beta}}_0 (A^{\alpha \beta}[u] - A^{\alpha \beta}[\ul u]) \\
 \geq \,& \frac{\tilde{c}}{2} + H [u] - H[\ul u]
\end{aligned}
\end{equation}
where $H [u] = \tF^{\alpha {\beta}}_0 A^{\alpha \beta} [u] - \psi [u]$.

Define
\[ \varPhi = - \eta \nabla_n (u - \ul u) + H [u] + Q \]
where $\eta = \tF^{\alpha {\beta}}_0 \sigma_{\alpha {\beta}}$ and
\[ Q \equiv \tF^{\alpha {\beta}}_0 \nabla_{\alpha {\beta}} \ul u
          - \tF^{\alpha {\beta}}_0 U_{\alpha {\beta}} (x_0) + \psi [u] (x_0). \]
From \eqref{c-210} and \eqref{c-220}
we see that $\varPhi (x_0) = 0$ and
$\varPhi \geq 0$ on $\partial M$ near $x_0$.

By \eqref{hess-E170} and assumption \eqref{A2} we have
\[  \begin{aligned}
\mathcal{L} H \leq \,& H_z [u] \mathcal{L} u
   + H_{p_k} [u] \mathcal{L} \nabla_k u
   + F^{ij} H_{p_k p_l} [u] \nabla_{ki} u \nabla_{lj} u + C \sum F^{ii} + C \\
   \leq \, & C \sum f_i + C \sum f_i |\lambda_i| + C.
   \end{aligned} \]
Therefore,
\begin{equation}
\label{gblq-B360}
 \begin{aligned}
\mathcal{L} \varPhi
  \leq   C \sum f_i + C \sum f_i |\lambda_i| + C.
\end{aligned}
\end{equation}

Consider the function $\varPsi$ defined in \eqref{hess-E176}.
Applying Lemma~\ref{ma-lemma-E10} and Proposition~\ref{g-prop1}
again for $A_1 \gg A_2 \gg A_3 \gg 1$ we derive
\begin{equation}
\label{cma-106}
  \left\{ \begin{aligned}
  & \mathcal{L} (\varPsi + \varPhi) \leq  0 \;\; \mbox{in $M_{\delta}$},  \\
        & \varPsi + \varPhi \geq 0 \;\; \mbox{on $\partial M_{\delta}$}.
\end{aligned} \right.
\end{equation}
By the maximum principle,
$\varPsi +  \varPhi \geq 0$ in $M_{\delta}$. Thus
$\nabla_n \varPhi (x_0) \geq - \nabla_n \varPsi (x_0) \geq -C$.

Write $u^t = t u + (1-t) \ul u$ and
\[ H [u^t] = \tF^{\alpha {\beta}}_0 A^{\alpha \beta} [u^t] - \psi [u^t].  \]
We have
\[  \begin{aligned}
H [u] - H[\ul u]
     = \,& \int_0^1 \frac{d H[u^t]}{dt} dt \\
     = \,& (u - \ul{u})  \int _0^1 H_z [u^t] dt
       + \sum \nabla_k (u - \ul{u}) \int _0^1 H_{p_k} [u^t] dt.
       \end{aligned}  \]
Therefore, at $x_0$,
\begin{equation}
\label{c-235}
 H [u] - H[\ul u]
     =  \nabla_n (u - \ul{u}) \int _0^1 H_{p_n} [u^t] dt
\end{equation}
and
\begin{equation}
\label{c-245}
\begin{aligned}
 \nabla_n H [u]
 = \,& \nabla_n H [\ul u]
        +  \sum \nabla_{kn} (u - \ul{u}) \int _0^1 H_{p_k} [u^t] dt \\
   \,& + \nabla_n (u-\ul{u}) \int _0^1 (H_z [u^t]
        + \nabla_n' H_{p_n} [u^t] + H_{z p_n} [u^t] \nabla_n u^t) dt \\
   \,& + \nabla_n (u-\ul{u}) \sum \int _0^1 H_{p_n p_l} [u^t] \nabla_{ln} u^t dt \\
  \leq \,&  \nabla_{nn} (u - \ul{u})
          \int _0^1 (H_{p_n} [u^t] + t H_{p_n p_n} [u^t] \nabla_n (u - \ul{u})) dt
           + C \\
   \leq \,&  \nabla_{nn} (u - \ul{u}) \int _0^1 H_{p_n} [u^t] dt + C
       \end{aligned}
 \end{equation}
since $H_{p_n p_n} \leq 0$, $\nabla_{nn} (u - \ul{u}) \geq 0$
and $\nabla_{n} (u - \ul{u}) \geq 0$.
It follows that
\begin{equation}
\label{c-255}
\begin{aligned}
 \nabla_n \varPhi (x_0)
   \leq \,& - \eta (x_0) \nabla_{nn} (x_0) +  \nabla_n H [u] (x_0) + C \\
   \leq \,& \Big(- \eta (x_0) + \int _0^1 H_{p_n} [u^t] (x_0) dt\Big)
               \nabla_{nn} u (x_0) + C.
       \end{aligned}
 \end{equation}
 By \eqref{c-225} and \eqref{c-235},
\begin{equation}
\label{c-230}
\eta (x_0) -  \int _0^1 H_{p_n} [u^t] (x_0) dt
                \geq \frac{\tilde{c}}{2 \nabla_n (u - \ul{u}) (x_0)}
                \geq \epsilon_1 \tilde{c} > 0
\end{equation}
for some uniform $\epsilon_1 > 0$.
This gives
\begin{equation}
\label{cma-310}
\nabla_{nn} u (x_0) \leq  \frac{C}{\epsilon_1 \tilde{c}}.
\end{equation}

So we have an {\em a priori}
upper bound for all eigenvalues of $\{U_{ij} (x_0)\}$.
Consequently, $\lambda [\{U_{ij} (x_0)\}]$ is contained in a
compact subset of $\Gamma$ by \eqref{3I-40}, and therefore
\[ m_R \equiv f (\lambda [U_{\alpha  {\beta}} (x_0)], R) 
     - \psi [u] (x_0) > 0 \]
when $R$ is sufficiently large.
This proves \eqref{cma-201} and  the proof of \eqref{hess-a10b}
is complete.

\bigskip

\section{The gradient estimates}
\label{gj-G}
\setcounter{equation}{0}

\medskip

In this section we consider the gradient estimates.
Throughout the section,
and in Theorems~\ref{gj-th1}-\ref{gj-th1b} below in particular,
we assume \eqref{3I-20}-\eqref{3I-40}, \eqref{A2}
and the following growth conditions hold
\begin{equation}
\label{A1}
\left\{ \begin{aligned}
    p \cdot \nabla_x A^{\xi \xi} (x, z, p) + |p|^2  A^{\xi \xi}_z (x, z, p)
         \,& \leq \bar{\psi}_1 (x, z) |\xi|^2 (1 + |p|^{\gamma_1}), \\
    p \cdot \nabla_x \psi (x, z, p)  + |p|^2 \psi_z (x, z, p)
         \,& \geq - \bar{\psi}_2 (x, z) (1 + |p|^{\gamma_2}),
  \end{aligned} \right.
\end{equation}
for some functions $\bar{\psi}_1, \bar{\psi}_2 > 0$ and constants
$\gamma_1, \gamma_2  > 0$.
Let $u \in C^3 (\bM)$  be an admissible solution of ~\eqref{3I-10}.

\begin{theorem}
\label{gj-th1}
Assume, in addition, that
\eqref{A3} and \eqref{gj-I105} hold, $\gamma_1 < 4$, $\gamma_2 = 2$
in \eqref{A1},
and that there is an admissible function $\ul u \in C^2 (\bM)$.
 Then
 \begin{equation}
\label{3I-R60}
 \max_{\bM} |\nabla u|
     \leq C_1 \big(1 + \max_{\partial M} |\nabla u|\big)
\end{equation}
where $C_1$ depends on $|u|_{C^0 (\bM)}$ and
$|\ul u|_{C^2 (\bM)}$.
\end{theorem}

\begin{proof}
Let $w = |\nabla u|$ and $\phi$ a positive function to be determined.
Suppose the function $w \phi^{-a}$
achieves a positive maximum at an interior point $x_0 \in M$
where $a < 1$ is constant.
Choose a smooth orthonormal local frame $e_1, \ldots, e_n$
about $x_0$ such that $\nabla_{e_i} e_j = 0$ at $x_0$
and $\{U_{ij} (x_0)\}$ is diagonal.


The function $\log w - a \log \phi$ attains its maximum
at $x_0$ where for $i = 1, \ldots, n$,
\begin{equation}
\label{g1} 
\frac{\nabla_i w}{w} - \frac{a \nabla_i \phi}{\phi} = 0,
\end{equation}
\begin{equation}
\label{g2}
\frac{\nabla_{ii} w}{w} + \frac{(a - a^2) |\nabla_{i} \phi|^2}{\phi^2}
   - \frac{a \nabla_{ii} \phi}{\phi} \leq 0.
\end{equation}
Next,
\[ w \nabla_i w = \nabla_{l} u \nabla_{il} u \]
and, by (\ref{hess-A70}) and (\ref{g1}),
\begin{equation}
\label{g3}
\begin{aligned}
w \nabla_{ii} w
    = \,& \nabla_{l} u \nabla_{iil} u + \nabla_{il} u \nabla_{il} u
          - \nabla_i w \nabla_i w \\
    = \,& (\nabla_{lii}u+ R^{k}_{iil} \nabla_{k} u) \nabla_{l} u
          + \Big(\delta_{kl} - \frac{\nabla_{k} u \nabla_{l} u}{w^2} \Big)
            \nabla_{ik} u \nabla_{il} u \\
 \geq \,& (\nabla_l U_{ii} - A^{ii}_{p_k} \nabla_{lk} u - A^{ii}_{u} \nabla_l u - \nabla'_l A^{ii}) \nabla_l u
          - C |\nabla u|^2 \\
    = \,& \nabla_{l} u \nabla_l U_{ii} - \frac{w^2}{\phi} (a A^{ii}_{p_{k}} \nabla_{k} \phi
          + \phi A^{ii}_{u}) - \nabla_l u \nabla'_l A^{ii} - C w^2.\\
\end{aligned}
\end{equation}
By \eqref{gj-S60} and (\ref{g1}),
\begin{equation}
\label{g11'}
\begin{aligned}
 F^{ii} \nabla_{l} u \nabla_l U_{ii}
    = \,& \nabla_{l} u \nabla'_{l} \psi + \psi_u |\nabla u|^2
          + \psi_{p_k} \nabla_{l} u \nabla_{lk} u \\
    = \,& \nabla_{l} u \nabla'_{l} \psi + \psi_u |\nabla u|^2
          + \frac{a w^2}{\phi} \psi_{p_k} \nabla_k \phi.
        \end{aligned}
\end{equation}

Let
$\phi = (u - \ul u) + b > 0$ where $b = 1 + \sup_{M} (\ul u - u)$.
By the MTW condition \eqref{A3} we have
\begin{equation}
\label{g3.5}
\begin{aligned}
  - A^{ii}_{p_{k}} \nabla_{k} \phi
      = \,&  A^{ii}_{p_{k}} (x, u, \nabla u) \nabla_{k} (\ul u - u) \\
   \geq \,& A^{ii} (x, u, \nabla \ul{u}) - A^{ii} (x, u, \nabla u)
             + c_0 (|\nabla \phi|^2 - |\nabla_i \phi|^2) \\
   \geq \,& A^{ii} (x, \ul u, \nabla \ul{u}) - A^{ii} (x, u, \nabla u)
             + c_0 (|\nabla \phi|^2 - |\nabla_i \phi|^2) - C.
\end{aligned}
\end{equation}
By (\ref{g2}), \eqref{g3} and (\ref{g3.5}),
\begin{equation}
\label{g10}
\begin{aligned}
0 \geq \,& \frac{\nabla_{l} u}{w^2} F^{ii} \nabla_l U_{ii}
           + \frac{a}{\phi} F^{ii} (\ul{U}_{ii}- U_{ii})
           + \frac{a c_0 |\nabla \phi|^2}{\phi} \sum F^{ii} \\
       \,& + \frac{a - a^2 - c_0 a \phi}{\phi^2} F^{ii} |\nabla_{i} \phi|^2
           - F^{ii} A^{ii}_u - \frac{\nabla_{l} u}{w^2} F^{ii}\nabla'_l A^{ii}
           - C \sum F^{ii}.
\end{aligned}
\end{equation}

Note that for $a \in (0, 1)$,
\begin{equation}
\label{g13}
 \frac{a c_0 |\nabla \phi|^2}{\phi} \sum F^{ii}
      + \frac{a - a^2 - c_0 a \phi}{\phi^2} F^{ii} |\nabla_{i} \phi|^2
\geq  c_0' |\nabla \phi|^2 \sum F^{ii}
\end{equation}
for some $c_0' > 0$.
 This is obvious if $c_0 \phi < 1$; if $c_0 \phi \geq 1$ then
\[  \frac{a c_0 |\nabla \phi|^2}{\phi} \sum F^{ii}
    + \frac{a - a^2 - c_0 a \phi}{\phi^2} F^{ii} |\nabla_{i} \phi|^2
\geq  \frac{a -a^2}{\phi^2}  |\nabla \phi|^2 \sum F^{ii}. \]
By (\ref{g11'}) and the convexity of $\psi (x, z, p)$
in $p$,
\begin{equation}
\label{g11}
\begin{aligned}
 F^{ii} \nabla_{l} u \nabla_l U_{ii}
    \geq \,& \nabla_{l} u \nabla'_{l} \psi + \psi_u |\nabla u|^2
      + \frac{a w^2}{\phi}  (\psi (x, u, \nabla u) - \psi (x, u, \nabla \ul u)).
        \end{aligned}
\end{equation}

Plugging \eqref{g13}, \eqref{g11} and \eqref{g12}
into \eqref{g10},
we derive for $B$ sufficiently large
\begin{equation}
\label{g14}
\begin{aligned}
 0 \geq \,& \frac{\nabla_{l} u \nabla'_{l} \psi}{w^2} + \psi_u
                   + \frac{a}{\phi} \Big(F(Bg)
                    - \psi (x, u, \nabla \ul u) - 2B \sum F^{ii}\Big)   \\
          & + c_0' |\nabla \phi|^2 \sum F^{ii} - F^{ii} A^{ii}_u
            - \frac{\nabla_{l} u}{w^2} F^{ii}\nabla'_l A^{ii} - C \sum F^{ii}.
\end{aligned}
\end{equation}
By \eqref{A1} we obtain
\begin{equation}
\label{g14'}
\begin{aligned}
 0 \geq \,& a F (Bg) - a \psi (x, u, \nabla \ul u)
       - C \phi |\nabla u|^{\gamma_2 - 2} \\
          & + (c_0' \phi |\nabla \phi|^2  - C \phi |\nabla u|^{\gamma_1 - 2}
              - C \phi - 2 a B) \sum F^{ii}.
\end{aligned}
\end{equation}
Since $\gamma_1 < 4$ and $\gamma_2 = 2$, by \eqref{gj-I105} this yields a bound $|\nabla u (x_0)| \leq C$ if
$B$ is chosen sufficiently large.
\end{proof}

\begin{theorem}
\label{gj-th1a}
Assume, in addition, that
{\bf (i)} $\psi = \psi (x, p)$, $A = A (x, p)$;
{\bf (ii)} $(M^n, g)$ has nonnegative sectional curvature; and
{\bf (iii)} \eqref{3I-11s}, \eqref{gj-I100} \eqref{A1} hold
for $\gamma_1, \gamma_2 < 2$ in \eqref{A1}, and
that there exist constants
$K > 0$ and $c_1 > 0$ such that
\begin{equation}
\label{A6}
\psi (x, p) \geq c_1, \;\; \forall \, x \in \bM, p \in T_x \bM; |p| \geq K.
\end{equation}
 Then \eqref{3I-R60} holds.
\end{theorem}

\begin{proof}
Since $(M, g)$ has nonnegative sectional curvature, in orthonormal local
frame,
\[ R^{k}_{iil} \nabla_{k} u \nabla_{l} u \geq 0. \]
In the proof of Theorem~\ref{gj-th1},  we therefore have in place of
\eqref{g3},
\begin{equation}
\label{g15}
\begin{aligned}
w \nabla_{ii} w
   \geq \,& \nabla_{lii} u \nabla_l u + R^{k}_{iil} \nabla_{k} u \nabla_{l} u \\
   \geq \,& \nabla_{l} u \nabla_l U_{ii}
           - \frac{w^2}{\phi} (a A^{ii}_{p_{k}} \nabla_{k} \phi
           + \phi A^{ii}_{u}) - \nabla_l u \nabla'_l A^{ii} .
\end{aligned}
\end{equation}
By \eqref{A1}, \eqref{g2}, \eqref{g15} and \eqref{g11}, we obtain
\begin{equation}
\label{g16}
\begin{aligned}
 0 \geq \,& - \frac{a}{\phi} \mathcal{L} \phi
                  + \frac{\nabla_{l} u \nabla'_{l} \psi}{w^2}
                  - \frac{1}{w^2} F^{ii} \nabla_l u \nabla'_l A^{ii}
               + \frac{(a - a^2)}{\phi^2} F^{ii} |\nabla_i \phi|^2\\
  \geq \,&  \frac{a}{\phi} \mathcal{L} (\ul u - u) + c_0F^{ii} |\nabla_i \phi|^2
              - C |\nabla u|^{\gamma_1 - 2}\sum F^{ii}
              - C |\nabla u|^{\gamma_2 - 2}.
\end{aligned}
\end{equation}

Suppose $|\lambda (U (x_0))| \geq R$ for $R$ sufficiently large.
As $\psi$ and $A$ are independent of $u$, by the comparison
principle $u \geq \ul u$ in $M$.
Consequently, we may apply Proposition~\ref{gj-lem3} to derive
a bound $|\nabla u (x_0)| \leq C$ from (\ref{g16}).

Suppose now that $|\lambda (U (x_0))| \leq R$ and
$|\nabla u (x_0)| \geq K$ for $K$ sufficiently large. Then there is
$C_2 > 0$ depending on $R$ and $K$ such that
\[ C_2^{-1} I \leq \{F^{ij}\} \leq C_2 I. \]
Since $\mathcal{L} (\ul u - u) \geq 0$, it follows from (\ref{g16})
that
$$
\begin{aligned}
 c_0 C_2^{-1} |\nabla \phi|^2
        - n C C_2 |\nabla u|^{\gamma_1-2} - C |\nabla u|^{\gamma_2-2} \leq 0.
\end{aligned}
$$
This proves $|\nabla u (x_0)| \leq C$.
\end{proof}

An alternative assumption
which is commonly used in deriving gradient estimate is the following
\begin{equation}
\label{3I-50}
f_{j} (\lambda) \geq \nu_{0} \Big(1 + \sum f_{i} (\lambda)\Big) \;\;
\mbox{if $\lambda_{j} < 0$}, \; \forall \, \lambda \in \Gamma^{\sigma}
\end{equation}
for any $\sigma > 0$ where $\nu_0  > 0$ depends on $\sigma$;
see e.g.
\cite{GS91}, \cite{Korevaar87}, \cite{LiYY91}, \cite{Trudinger90},
 and \cite{Urbas02}.

\begin{theorem}
\label{gj-th1b}
Assume, in addition, that $\gamma_1, \gamma_2 < 4$,
\eqref{3I-45}, \eqref{3I-50} hold, and that
\begin{equation}
\label{A5}
- \psi_z (x, z, p), \; p \cdot D_p \psi (x, z, p),
    \; - p \cdot D_p A^{\xi \xi} (x, p)/|\xi|^2
     \leq \bar{\psi}(x, z)  (1 + |p|^{\gamma}),
\end{equation}
\begin{equation}
\label{gj-G20**}
|A^{\xi \eta} (x, z, p)|
     \leq \bar{\psi} (x, z) |\xi||\eta| (1 + |p|^{\gamma}),
\;\; \forall \, \xi, \eta \in T^*\bM; \xi \perp \eta.
\end{equation}
for some function $\bar{\psi} \geq 0$ and constant $\gamma \in (0, 2)$.
There exists a constant $C_1$ depending on $|u|_{C^0 (\bM)}$ and other
known data such that \eqref{3I-R60} holds.
\end{theorem}

\begin{proof}
In the proof of Theorem~\ref{gj-th1} we take
$\phi = - u + \sup_M u + 1$.
By the concavity of $A^{ii} (x, z, p)$ in $p$ we have
\begin{equation}
\label{gj-G20}
 A^{ii} = A^{ii} (x, u, \nabla u)
\leq A^{ii} (x, u, 0) + A^{ii}_{p_k} (x, u, 0) \nabla_k u.
\end{equation}
By assumption~\eqref{3I-45},
\begin{equation}
\label{g17}
- F^{ii} \nabla_{ii} \phi = F^{ii} \nabla_{ii} u = F^{ii} U_{ii} - F^{ii} A^{ii}
     \geq -F^{ii} A^{ii} \geq - C (1 + |\nabla u|) \sum F^{ii}.
\end{equation}
It follows from \eqref{g2}, \eqref{g3}, \eqref{g11'},
\eqref{g17}, \eqref{A1} and \eqref{A5} that for $a < 1$,
\begin{equation}
\label{g18}
\begin{aligned}
0 \geq & \frac{(a - a^2)}{\phi^2} F^{ii} |\nabla_{i} u|^2
       + \frac{\nabla_{l} u \nabla'_{l} \psi}{w^2} + \psi_u
           - \frac{a}{\phi} \psi_{p_k} \nabla_k u  \\
        & + \frac{a}{\phi} F^{ii} A^{ii}_{p_{k}} \nabla_{k} u
           -  F^{ii} A^{ii}_{u} - F^{ii} \frac{\nabla_l u \nabla'_l A^{ii}}{w^2}
            - C (1 + |\nabla u|)\sum F^{ii} \\
  \geq & c_0 F^{ii} |\nabla_{i} u|^2 - C |\nabla u|^{\gamma - 2}
             - C(1 + |\nabla u| + |\nabla u|^{\gamma - 2})  \sum F^{ii}.
\end{aligned}
\end{equation}

Without loss of generality we assume
$\nabla_{1} u (x_{0}) \geq \frac{1}{n} |\nabla u (x_{0})| > 0$.
Recall that $U_{ij} (x_0)$ is diagonal. By \eqref{g1}, \eqref{gj-G20}
and \eqref{gj-G20**} we derive
\begin{equation}
\label{gj-G30}
U_{11} = - \frac{a}{\phi} |\nabla u|^2 + A_{11}
 + \frac{1}{\nabla_1 u} \sum_{k \geq 2} \nabla_k u A^{1k}
 \leq - \frac{a}{\phi} |\nabla u|^2
      + C (1 + |\nabla u| + |\nabla u|^{\gamma-2}).
\end{equation}

If $U_{11} (x_0) \geq 0$ we obtain a bound $|\nabla u (x_0)| \leq C$ from
\eqref{gj-G30}. If $U_{11} (x_0) < 0$ then by \eqref{3I-50},
\[ f_{1} \geq \nu_{0} \Big(1 + \sum f_{i}\Big) \]
and a bound $|\nabla u (x_0)| \leq C$ follows from \eqref{g18}.
\end{proof}

\bigskip

\section{The Dirichlet Problem}
\label{gj-E}
\setcounter{equation}{0}

\medskip

We now turn to existence of solutions to the Dirichlet
problem~\eqref{3I-10} and \eqref{3I-10d}.
We first consider the special case
$A= A (x, p)$ and $\psi= \psi (x, p)$.

\begin{theorem}
\label{gj-th3}
Suppose $(M^n, g)$ is a compact Riemannian manifold of nonnegative sectional
curvature with smooth boundary $\partial M$,
$A = A (x, p)$ and $\psi = \psi (x, p)$ are smooth,  and
$\varphi \in C^{\infty} (\partial M)$.
Assume that \eqref{3I-20}-\eqref{3I-40},
\eqref{A2}, \eqref{3I-11s}, \eqref{3I-45}, \eqref{A1} and
\eqref{A6} hold for $\gamma_1, \gamma_2 < 2$ in \eqref{A1}.
Then there exists an admissible solution $u \in C^{\infty} (\bar{M})$ of
equation~\eqref{3I-10} satisfying the boundary condition \eqref{3I-10d}.
\end{theorem}

As $A$ and $\psi$ are assumed to be independent of $u$
in Theorem~\ref{gj-th3}, by the maximum principle it is easy to derive
the estimate
\begin{equation}
\label{gj-I115}
 \max_{\bM} |u| + \max_{\partial M} |\nabla u| \leq C.
\end{equation}
By Theorems~\ref{gj-th2} and \ref{gj-th1a}
we obtain
\begin{equation}
\label{gj-I120}
|u|_{C^2 (\bM)} \leq C.
\end{equation}
From \eqref{3I-40} and the fact that $\psi > 0$  we see that
equation~\eqref{3I-10} becomes
uniformly elliptic for admissible solutions satisfying
\eqref{gj-I120}.
Consequently, the concavity condition \eqref{3I-30} allows us to apply
Evans-Krylov theorem in order to obtain
$C^{2, \alpha}$ estimates; and higher order estimates follow from the
Schauder theory.
Theorem~\ref{gj-th3} may then be proved using the standard continuity
method.

\begin{theorem}
\label{gj-th5}
Let $(M^n, g)$ be a compact Riemannian manifold with smooth
boundary $\partial M$. Suppose $A = A (x, p)$, $\psi = \psi (x, z, p)$ are
smooth and $\varphi \in C^{\infty} (\partial M)$.
In addition to \eqref{3I-20}-\eqref{3I-40}, \eqref{A2}-\eqref{3I-11s},
\eqref{3I-45}, \eqref{A1}, \eqref{3I-50} and \eqref{A5}, assume that
\begin{equation}
\label{gj-G20*}
|A^{\xi \eta} (x, p)|
     \leq \bar{\psi} (x) |\xi||\eta| (1 + |p|^{\gamma}),
\;\; \forall \, \xi, \eta \in T^*\bM, \xi \perp \eta
\end{equation}
for some function $\bar{\psi} \geq 0$ and constant $\gamma \in (0, 2)$.
Then the Dirichlet problem~\eqref{3I-10} and \eqref{3I-10d} admits
an admissible solution $u \in C^{\infty} (\bar{M})$ satisfying $u \geq \ul u$
on $\bM$.
\end{theorem}


\bigskip

\small
\begin{thebibliography}{99}

\bibitem{Alexandrov56}
A. D. Alexandrov,
{\em Uniqueness theorems for surfaces in the large, I},
Vestnik Leningrad. Univ. {\bf 11} (1956), 5-17.

\bibitem{Andrews94}
B. Andrews,
{\em Contraction of convex hypersurfaces in Euclidean space},
Calc. Var. PDE  {\bf 2} (1994), 151--171.

\bibitem{CNS1}
L. A. Caffarelli, L. Nirenberg and J. Spruck,
{\em The Dirichlet problem for nonlinear second-order elliptic equations I.
Monge-Amp\`{e}re equations},
{ Comm. Pure Applied Math.} {\bf 37} (1984), 369--402.

\bibitem{CNS3}
L. A. Caffarelli, L. Nirenberg and J. Spruck,
{\em The Dirichlet problem for nonlinear second-order elliptic equations
 III: Functions of eigenvalues of the Hessians},
{ Acta Math.} {\bf 155} (1985), 261--301.


\bibitem{CGY02a}
A. Chang, M. Gursky and P. Yang,
{\em An equation of Monge-Amp\`ere type in conformal geometry, and
four-manifolds of positive Ricci curvature},
Ann. Math. (2) {\bf 155} (2002), 709-787.


\bibitem{ChenS05}
S.-Y. S. Chen,
{\em Local estimates for some fully nonlinear elliptic equations},
Int. Math. Res. Not. {\bf 2005} (2005) no. 55, 3403--3425.

\bibitem{CY76b}
S. Y. Cheng and S. T. Yau,
{\em On the regularity of the solution of the n-dimensional Minkowski problem},
Comm. Pure Applied Math. {\bf 29} (1976), 495--516.

\bibitem{Chern59}
S. S. Chern,
{\em Integral formulas for hypersurfaces in Euclidean space and their
applications to uniqueness theorems},
J. Math. Mech., {\bf 8} (1959), 947--955.

\bibitem{CW01}
K.-S. Chou and X.-J. Wang,
{\em A variational theory of the Hessian equation},
Comm. Pure Appl. Math. {\bf 54} (2001), 1029--1064.

\bibitem{GeW06}
Y.-X. Ge and G.-F. Wang,
{\em On a fully nonlinear Yamabe problem},
Ann. Sci. École Norm. Sup. (4) {\bf 39} (2006), 569--598.


\bibitem{Gerhardt96}
C. Gerhardt,
{\em Closed Weingarten hypersurfaces in Riemannian manifolds},
J. Differential Geometry  {\bf 43} (1996), 612--641.

\bibitem{Guan94}
B. Guan,
{\em The Dirichlet problem for a class of fully nonlinear elliptic equations},
{ Comm. Partial Diff. Equations} {\bf 19} (1994), 399--416.

\bibitem{Guan98a}
B. Guan,
{\em The Dirichlet problem for Monge-Amp\`ere equations in non-convex
domains and spacelike hypersurfaces of constant Gauss curvature},
Trans. Amer. Math. Soc. {\bf 350} (1998), 4955--4971.

\bibitem{Guan99a}
B. Guan,
{\em The Dirichlet problem for Hessian equations on Riemannian manifolds},
Calc. Var. PDE {\bf 8} (1999) 45--69.

\bibitem{Guan12a}
B. Guan,
{\em Second order estimates and regularity for fully nonlinear ellitpic
equations on Riemannian manifolds}, to appear in Duke Math. J.

\bibitem{Guan}
B. Guan,
{\em The Dirichlet problem for Hessian equations in general domains},
preprint.

\bibitem{GG02}
B. Guan and P.-F. Guan,
{\em Closed hypersurfaces of prescribed curvatures},
Ann. Math. (2) {\bf 156} (2002) 655--673.

\bibitem{GS91}
B.~Guan and J.~Spruck,
{\em Interior gradient estimates for solutions of prescribed curvature
equations of parabolic type},
Indiana Univ. Math. J. {\bf 40} (1991) 1471--1481.

\bibitem{GS93}
B.~Guan and J.~Spruck,
{\em Boundary value problem on $\bfS^n$ for surfaces of constant Gauss
 curvature},
{ Ann. Math.}  (2) {\bf 138} (1993), 601--624.



\bibitem{GLL12}
P.-F. Guan, J.-F. Li and Y.-Y. Li,
{\em Hypersurfaces of prescribed curvature measures},
Duke Math. J. {\bf 161} (2012), 1927--1942.

\bibitem{GL94}
P.-F. Guan and Y.-Y. Li,
{\em On Weyl problem with nonnegative Gauss curvature},
J. Differential Geometry {\bf 39} (1994), 331--342.

\bibitem{GL97}
P.-F. Guan and Y.-Y. Li,
{\em  $C^{1,1}$ Regularity for solutions of a problem of Alexandrov},
Comm. Pure Applied Math. {\bf 50} (1997), 789--811.

\bibitem{GM03}
P.-F. Guan and X.-N. Ma,
{\em The Christoffel-Minkowski problem. I. Convexity of solutions of a
Hessian equation},
Invent. Math. {\bf 151} (2003), 553--577.

\bibitem{GRW13}
P.-F. Guan, C.-Y. Ren and Z.-Z Wang,
{\em Global $C^2$ estimates for convex solutions of curvature equations},
to apear in Comm. Pure Applied Math.

\bibitem{GW03a}
P.-F. Guan and G.-F. Wang,
{\em Local estimates for a class of fully nonlinear equations arising from
conformal geometry},
Int. Math. Res. Not. {\bf 2003}, no.26, 1413--1432.

\bibitem{GW03b}
P.-F. Guan and G.-F. Wang,
{\em A fully nonlinear conformal flow on locally conformally flat manifolds}
J. Reine Angew. Math. {\bf 557}  (2003), 219--238.

\bibitem{GW98}
P.-F. Guan and X.-J. Wang,
{\em On a Monge-Amp\`ere equation arising in geometric optics},
J. Differential Geometry {\bf 48} (1998), 205--222.


\bibitem{GV07}
M. J. Gursky and J. A. Viaclovsky,
{\em  Prescribing symmetric functions of the eigenvalues of the Ricci
tensor},
Ann. Math. (2) {\bf 166} (2007), 475--531.

\bibitem{HZ95}
J.-X. Hong and C. Zuily,
{\em Isometric embedding of the 2-sphere with nonnegative curvature in $\bfR^3$},
Math. Z. {\bf 219} (1995), 323--334.


\bibitem{Ivochkina80}
N. M.~Ivochkina,
{\em The integral method of barrier functions and the Dirichlet problem
for equations with operators of the Monge-Amp\`ere type},
(Russian) Mat. Sb. (N.S.) 112 (1980), 193–-206;
English transl.: Math. USSR Sb. {\bf 40} (1981) 179–-192.

\bibitem{Ivochkina85}
N. M.~Ivochkina,
{\em Solution of the Dirichlet problem for certain equations of
 Monge-Amp\`ere type},
Mat. Sb. (N.S.) {\bf 128 (170)} (1985), 403--415.


\bibitem{ITW04}
N. M. Ivochkina, N. Trudinger and  X.-J. Wang,
{\em The Dirichlet problem for degenerate Hessian equations},
Comm. Partial Diff. Equations {\bf 29} (2004), 219--235.

\bibitem{Korevaar87}
N. J. Korevaar,
{\em A priori gradient bounds for solutions to elliptic Weingarten equations},
Ann. Inst. Henri Poincar\'e, Analyse Non Lin\'eaire {\bf 4} (1987), 405--421.

\bibitem{Krylov83}
N. V.~Krylov,
{\em Boundedly nonhomogeneous elliptic and parabolic equations in a domain},
{Izvestia Math. Ser.} {\bf 47} (1983), 75--108.

\bibitem{LL03}
A.-B. Li, and Y.-Y. Li,
{\em On some conformally invariant fully nonlinear equations},
Comm. Pure Appl. Math. {\bf 56} (2003), 1416--1464.

\bibitem{LL05}
A.-B. Li and Y.-Y. Li,
{\em On some conformally invariant fully nonlinear equations.
Part II: Liouville, Harnack and Yamabe},
Acta Math. {\bf 195} (2005), 117--154.



\bibitem{LiYY90}
Y.-Y. Li,
{\em Some existence results of fully nonlinear elliptic equations
of Monge-Amp\`ere type},
{ Comm. Pure Applied Math.}  {\bf 43} (1990), 233--271.

\bibitem{LiYY91}
Y.-Y. Li,
{\em Interior gradient estimates for solutions of certain
fully nonlinear elliptic equations},
J. Diff. Equations {\bf 90} (1991), 172--185.


\bibitem{MTW05}
X.-N. Ma, N. S. Trudinger, X.-J. Wang,
{\em Regularity of potential functions of the optimal transportation problem},
Arch. Rational Mech. Anal. {\bf 177} (2005), 151--183.

\bibitem{Nirenberg53}
L. Nirenberg,
{\em The Weyl and Minkowski problems in differential geometry in the large},
Comm. Pure Applied Math. {\bf 6} (1953), 337-394.

\bibitem{Pogorelov52}
A. V. Pogorelov,
{\em Regularity of a convex surface with given Gaussian curvature},
Mat. Sb.  {\bf 31} (1952), 88-103.

\bibitem{Pogorelov78}
A. V. Pogorelov,
{\em The Minkowski Multidimentional Problem},
Winston, Washington, 1978.

\bibitem{STW07}
W.-M. Sheng, N. Trudinger and X.-J. Wang,
{\em The Yamabe problem for higher order curvatures},
J. Differential Geometry {\bf 77} (2007), 515--553.

\bibitem{SUW04}
W.-M. Sheng, J. Urbas and X.-J. Wang,
{\em Interior curvature bounds for a class of curvature equations},
Duke J. Math. {\bf 123} (2004), 235--264.

\bibitem{Trudinger90}
N. S. Trudinger,
{\em The Dirichlet problem for the prescribed curvature equations},
Arch. National Mech. Anal. {\bf 111} (1990) 153--179.

\bibitem{Trudinger95}
N. S. Trudinger,
{\em On the Dirichlet problem for Hessian equations},
Acta Math. {\bf 175} (1995), 151--164.

\bibitem{Trudinger06}
N. S. Trudinger,
{\em Recent developments in elliptic partial differential equations
of Monge-Amp\`ere type}, ICM Madrid {\bf 3} (2006), 291--302.

\bibitem{TW99}
N. S. Trudinger and X.-J. Wang,
{\em Hessian measures. II.},
Ann. Math. (2) {\bf 150} (1999), 579–-604.

\bibitem{Urbas02}
J. Urbas,
{\em Hessian equations on compact Riemannian manifolds},
Nonlinear Problems in Mathematical Physics and Related Topics II 367--377,
Kluwer/Plenum, New York, 2002.

\bibitem{Viaclovsky00}
J. A. Viaclovsky,
{\em Conformal geometry, contact geometry, and the calculus of variations},
Duke Math. J.  {\bf 101} (2000), 283--316.

\bibitem{Wang94}
X.-J. Wang,
{\em A class of fully nonlinear elliptic equations and related functionals},
Indiana Univ. Math. J. {\bf 43} (1994), 25--54.

\end{thebibliography}
\end{document}